\documentclass[a4paper,oneside,11 pt]{article}
 
\usepackage[T1]{fontenc}
\usepackage[utf8]{inputenc}
\usepackage{amsthm}
\usepackage{amsmath}
\usepackage{amssymb}
\usepackage{mathrsfs}
\usepackage{graphicx}
\usepackage{enumitem} 
\usepackage{tikz}
\usepackage{natbib}

\author{Léonard Cadilhac}

\theoremstyle{theorem}
\newtheorem{definition}{Definition}[section]
\newtheorem{theorem}[definition]{Theorem}
\newtheorem{property}[definition]{Proposition}
\theoremstyle{definition}
\newtheorem{remark}[definition]{Remark}

\newtheorem{lemma}[definition]{Lemma}
\theoremstyle{theorem}
\newtheorem{corollary}[definition]{Corollary}

\newcommand{\Rb}{\mathbb{R}}
\newcommand{\Mb}{\mathbb{M}}

\newcommand{\Nb}{\mathbb{N}}

\newcommand{\Cb}{\mathbb{C}}

\newcommand{\Hs}{\mathscr{H}}

\newcommand{\A}{\mathcal{A}}
\newcommand{\B}{\mathcal{B}}

\newcommand{\M}{\mathcal{M}}
\newcommand{\N}{\mathcal{N}}
\newcommand{\E}{\mathcal{E}}

\newcommand{\F}{\mathcal{F}}
\newcommand{\Pc}{\mathcal{P}}

\newcommand\bbox[1]{\left[ #1 \right]}

\newcommand\Norm[1]{\big\| #1 \big\|}
\newcommand\md[1]{\left| #1 \right|}

\newcommand\Floor[1]{\left \lfloor {#1} \right \rfloor}

\newcommand\p[1]{\left( #1 \right)}

\newcommand\Minf{\M_{\infty}^{-}}

\newcommand{\les}{\lesssim}

\newcommand{\e}{\varepsilon}

\newcommand{\Ind}{\mathbf{1}}

\newcommand\supb[1]{\underset{#1}{\sup}}
\newcommand\Sum[2]{\sum\nolimits_{#1}^{#2}}
\newcommand\Int[2]{\int\limits_{#1}^{#2}}

\newcommand{\CZ}{\text{Calder\'on-Zygmund}}

\newcommand\addtag{\refstepcounter{equation}\tag{\theequation}}

\title{Noncommutative Khintchine inequalities in interpolation spaces of $L_p$-spaces}
\author{Léonard Cadilhac}

\begin{document}

\maketitle

\begin{abstract}
We prove noncommutative Khintchine inequalities for all interpolation spaces between $L_p$ and $L_2$ with $p<2$. In particular, it follows that Khintchine inequalities hold in $L_{1,\infty}$. Using a similar method, we find a new deterministic equivalent for the $RC$-norm in all interpolation spaces between $L_p$-spaces which unifies the cases $p > 2$ and $p < 2$. It produces a new proof of Khintchine inequalities for $p<1$ for free variables. To complete the picture, we exhibit counter-examples which show that neither of the usual closed formulas for Khintchine inequalities can work in $L_{2,\infty}$. We also give an application to martingale inequalities.
\end{abstract}

\thanks{\textbf{Keywords:} Khintchine inequalities, noncommutative integration, symmetric spaces.}

\section{Introduction}

This paper is intended as a step towards completing the study of noncommutative Khintchine inequalities in interpolation spaces between $L_p$-spaces. No satisfying results were known in $L_{1,\infty}$ and $L_{2,\infty}$ in spite of the extensive literature  on the subject which includes some variants in general symmetric spaces. The remarkable growth of this topic in the last decades is to be attributed to the central role Khintchine inequalities play in noncommutative analysis. Similarly to their classical counterpart, they appear constantly when the norm of an unconditional sequence has to be estimated, and they allow us to describe the Banach space structure of the span of independent or free random variables. They were a stepping stone to develop noncommutative martingale inequalities which are essential and powerful tools to translate classical notions to the noncommutative setting. 

The seminal result of Lust-Piquard \cite{Lus86} and then Lust-Piquard and Pisier \cite{LusPis91} who first fomulated and proved Khintchine inequalities in the setting of noncommutative integration (for Rademacher variables and in $L_p$-spaces) led to generalisations spreading into different directions. In the context of free probability or analysis on the free group, they were introduced by Pisier and Haagerup in \cite{HaaPis93}, and studied further in different works (see for example \cite{ParPis05} and \cite{RicXu06}). As mentioned before, Khintchine inequalities are also the precursors of noncommutative martingale inequalities (\cite{PisXu97}, \cite{Jun02}, \cite{JunXu03}), another pillar of the theory, see for example \cite{JunXu08}, \cite{Dir15} and \cite{Narcisse2015}. Closely related to our subject, for more than a decade, attention has been given towards general symmetric spaces. One can mention, for example, the work of Lust-Piquard and Xu (\cite{LusXu07}), Le Merdy and Sukochev (\cite{LeMSuk08}) and Dirksen, de Pagter, Potapov, Sukochev (\cite{DirSuk11}). But only recently the case of quasi-Banach spaces was tackled by Pisier and Ricard in \cite{PisierRicard2017} who proved Khintchine inequalities in noncommutative $L_p$-spaces for $p<1$. The latter paper gives the final key inequality to apply the method found in \cite{Pis09} by Pisier. Our method takes inspiration from \cite{RicardDirksen} where Dirksen and Ricard proved the upper Khintchine inequalities in a very efficient way. We give a similar proof of the lower Khintchine inequality which is usually obtained by duality. This partially explains the difficulty of proving Khintchine inequalities in $L_p$-spaces for $p<1$ or in any quasi-Banach space. 

We present two different results (Sections \ref{section:projection} and \ref{Section alpha}) with independent proofs though they partially rely on the same idea. In the first one, we show that the lower Khintchine inequality in $L_p$ for $p<2$ implies the lower Khintchine inequalities for all interpolation spaces between $L_p$ and $L_{\infty}$ with a decomposition that does not depend on the space. This, combined with known results, directly implies Khintchine inequalities in $L_{1,\infty}$, which could not be reached before due to the inapplicability of interpolation or duality techniques in this case. A motivation to prove this last result was that it allows us to prove the weak-$1$-boundedness of $\CZ$ operators in the noncommutative setting for Hilbert-valued kernels (\cite{MeiPar09}) directly from the scalar-valued kernel case (\cite{Par09}), see \cite{Cad18}.

The second theorem gives a deterministic equivalent for "free averages" in every interpolation spaces between $L_p$-spaces. The remarkable feature here is that its formulation does not depend on whether $p<2$ or $p>2$. In particular, it holds in $L_{2,\infty}$ which is a tricky case since neither of the two usual formulas for Khintchine inequalities work (see Section \ref{section:counterexample}). Note that a deterministic formula was already found using interpolation methods  by Pisier in \cite{Pis11}. Our equivalent is less tractable than the usual formulas. It is obtained by first proving that any sequence of operators $(x_i)$ admits a factorisation of the form $\alpha (u_i) + (u_i) \beta$ where $\alpha$ and $\beta$ are positive operators and $(u_i)$ has good properties. Then, if $(\xi_i)$ is a sequence of free Haar unitaries, the norm of $\Sum{}{}x_i\otimes \xi_i$ happens to coincide with the norm of $\alpha \oplus \beta$ in all interpolation spaces between $L_p$-spaces. This yields a new proof of free Khintchine inequalities for $p<1$. It does not apply to Rademacher variables since it relies on Haagerup's inequality (\cite{Haa78}) together with a H\"older inequality for anti-commutators found in \cite{Ric18}.

In Section \ref{section:preliminaries}, we give a brief introduction to noncommutative analysis and introduce the tools we need. This allows us to precisely state our main theorems. We go into more details to prove some properties of the $K$-functional. Though most are well-known to the community, we were not successful in finding a published reference for them. In Section \ref{section:martingale inequalities}, we give an application of the first result to noncommutative martingale inequalities and in section \ref{section:technical lemmas}, we prove some technical lemmas needed in the core of the paper. 

\section{Preliminaries} \label{section:preliminaries}

\subsection{Noncommutative integration}

In this section we briefly recall definitions of some of the main objects appearing in noncommutative integration. We will suppose that $\M$ is a von Neumann algebra with a semifinite normal nonnegative faithful trace $\tau$. Together, they form a \emph{noncommutative measure space}. For any $p \in (0,\infty)$, we can define the $L_p$-norm (or quasi-norm) on this space by Borel functional calculus and the following formula:
$$\Norm{x}_p = \tau (\md{x}^p)^{1/p} .$$
The completion of $\{ x \in \M : \Norm{x}_p < \infty \}$ with respect to $\Norm{.}_p$ is denoted by $L_p(\M)$ and verifies properties similar to those of classical $L_p$-spaces. 

The noncommutative analog of measurable functions is denoted by $L_0(\M)$ and is the space of unbounded operators $x$ affilated with $\M$ which are $\tau$-measurable i.e there exists $\lambda \in \Rb^+$ such that $\tau(\Ind_{(\lambda,\infty)}(\md{x})) < \infty$ (where $\Ind_{(\lambda,\infty)}(\md{x}))$ is defined by functional calculus). This space $L_0(\M)$ continuously contains $L_p(\M)$ for all $p$. This makes any couple $(L_p(\M)$, $L_q(\M))$ compatible in the sense of interpolation.

The \emph{support} of any self-adjoint element $x\in L_0(\M)$ is defined as the least projection $s(x)$ such that $s(x)x = x$. Let $$\M_c = \{ x\in \M : \tau(s(\md{x})) < \infty \}$$ be the space of bounded, finitely supported operators in $\M$. For any $p\in (0,\infty)$, $\M_c$ is dense in $L_p(\M)$.
 
Denote by $\Pc(\M)$ the set of orthogonal projections in $\M$ and by $\Pc_c(\M)$ the set of finite projections in $\M$. A crucial tool to understand and study those spaces is the {\it generalised singular numbers} $\mu(x)$ associated to any $x\in L_0(\M)$. They can be defined by the following formula:
\begin{align*}
\mu(x) : & \Rb^+ \to \Rb^+ \\
& t \mapsto \mu_t(x) = \inf \{ \Norm{ex}_{\infty} : e\in \Pc(\M), \tau(1-e) \leq t \}, \\
&\ \ \ \ \ \ \ \ \ \ \ \ \ = \sup \{ a\in\Rb^+ : \tau(\Ind_{(a,\infty)}(\md{x})) \geq t \}.
\end{align*}
This formula may not be enlightning but $\mu(x)$ is to be thought as a nonincreasing nonnegative function which has the same distribution as $x$, in particular $\Norm{\mu(x)}_p = \Norm{x}_p$ for all $p$. Recall also the homogeneity property of $\mu$: for any $x\in L_0(\M)$, and $p\in (0,\infty)$, $\mu(x^p) = \mu(x)^p$.

\subsection{Symmetric Spaces, Interpolation}

Symmetric spaces (see \cite{KrePetSem82}) generalise $L_p$-spaces and can also be defined in the noncommutative setting (\cite{KalSuk08}). If $E$ is a symmetric function space on $\Rb^+$ equipped with the norm $\Norm{.}_E$, then $E(\M)$ is the space of all $x\in L_0(\M)$ such that $\mu(x)\in E$ equipped with the norm
$$\Norm{x}_{E(\M)} = \Norm{\mu(x)}_{E}.$$
This paper only deals with symmetric spaces which are interpolation spaces between $L_p$-spaces. The interpolation methods developped in the classical setting translate very well to noncommutative analysis and are some of the main techniques constantly used in the field. In particular, the noncommutative Lorentz spaces can be defined and keep their interpolation related properties (see \cite{Xu_notes}). For a general introduction to interpolation see \citep{BerghLofstrom}. Let us recall the definition of an interpolation space.

\begin{definition}
Let $(A,B)$ be a compatible couple of quasi-Banach spaces. We say that a quasi-Banach space $E$ is an interpolation space between $A$ and $B$ if $A\cap B \subset E \subset A+B$ and there exists a constant $C>0$ such that for every bounded operator $T:A+B \to A+B$ such that its restriction to $A$ (resp. $B$) if bounded of norm $1$ from $A$ to $A$ (resp. $B$ to $B$), $T$ is bounded from $E$ to $E$ with norm less than $C$.
\end{definition}

We will only use one notion from the theory of interpolation, the $K$-functional. Recall that it is defined as follows. Let $A,B$ be two quasi-Banach spaces, $x\in A+B$ and $t>0$ then:
$$K_t(x,A,B) := \inf \{ \Norm{y}_A + t\Norm{z}_B : y \in A, z\in B, y+z = x \}.$$
We will come back to this expression in the last subsection of the preliminaries. Until then, let us only mention the  following result which will enable us to use estimates on the $K$-functional to obtain inequalities for norms in general interpolation spaces. The difficult part of its proof is the commutative case which is dealt with in \cite{Spa78} except for the case $p<1$ and $q=\infty$ which is proved in \cite{Cad18Maj} (note that the case of sequence spaces is treated in \cite{Cwi81}). The fact that it still holds for noncommutative $L_p$-spaces is a direct consequence of proposition \ref{prop:Kt_independant of M}. 
\begin{property} \label{prop:kfuncandinterpolation}
Let $p,q\in (0,\infty]$, let $E$ be an interpolation space between $L_p(0,\infty)$ and $L_q(0,\infty)$ then there exists a constant $C$ such that for any $x,y \in E(\M)$, if $K(x,L_p(\M),L_q(\M)) \leq K(y,L_p(\M),L_q(\M))$ then
$$\Norm{x}_{E(\M)} \leq C \Norm{y}_{E(\M)}.$$
\end{property}
In the main sections of the paper, we find deterministic estimates of some $K$-functionals. They translate to Khintchine-type inequalities by means of this proposition. 

\subsection{Noncommutative Khintchine inequalities} \label{section:noncommutative khintchine inequalities}

Let us now introduce the general framework of noncommutative Khintchine inequalities. Denote by $S(\M)$ the set of finite sequences of elements of $\M_c$. Note that for $a,b\in \M$ and $x\in S(\M)$ we can naturally define the product $axb$ by:
$$axb = (ax_nb)_{n\geq 0}.$$
Let $x\in S(\M)$. Denote by $N\in\Nb$ its length and consider the projection,
$$e = s\p{\Sum{i=0}{N}x_i^*x_i} \vee s\p{\Sum{i=0}{N}x_ix_i^*}.$$
The fact that $x\in (e\M e)^N$ will be useful in various proofs in the paper since it enables us to work with only finite sequences of finitely supported elements.

Consider $(\A,\tau_{\A})$ another noncommutative probability space ($\tau_{\A}(1) = 1$) and $\xi = (\xi_i)_{i\in\Nb}$ a sequence of elements in $\A$. Recall that the elements of $\B(\ell^2)$ can be identified with infinite matrices and that $\B(\ell^2)$ is endowed with its canonical trace. We will denote by $e_{n,m}$ the element $(\delta_{i,n}\delta_{j,m})_{i,j\in\Nb}\in\B(\ell^2)$. 
For $x$ in $S(\M)$, we define $Rx = \Sum{n\geq 0}{} x_n \otimes e_{1,n} $, $Cx = \Sum{n\geq 0}{} x_n \otimes e_{n,1}$ which are understood as elements of the von Neumann algebra $\M_{\B} := \M \overline{\otimes} \B(\ell^2)$ and $Gx = \Sum{n\geq 0}{} x_n \otimes \xi_n$ in $\M_{\A} := \M \overline{\otimes} \A$. Both algebras are equipped with the tensor product trace. Throughout the paper, $\M$ will be identified with $\M \otimes e_{1,1}$ in $\M_{\B}$ and with $\M \otimes 1$ in $\M_{\A}$. With this in mind, note that:
\begin{center}
$\md{(Rx)^*} = \p{\Sum{n\geq 0}{}x_nx_n^*}^{1/2}$ and $\md{Cx} = \p{\Sum{n\geq 0}{} x_n^*x_n}^{1/2}$.
\end{center}

Fix $E$ a symmetric space. The quantity we want to estimate is $\Norm{x}_{\Hs_E} := \Norm{Gx}_{E(\M_{\A})}$ which is a quasi-norm on $S(\M)$. Denote by $\Hs_E(\M)$ ($\Hs_E$ if there is no ambiguity) the completion of $S(\M)$ for $\Norm{.}_{\Hs_E}$. Similarly, define $\Norm{x}_{R_E} := \Norm{Rx}_{E(\M_{\B})}$ (resp. $\Norm{x}_{C_E} := \Norm{Cx}_{E(\M_{\B})}$) and $R_E$ (resp. $C_E$) its completion. To lighten the notations, for $p\in (0,\infty]$, we write $\Hs_p := \Hs_{L_p}$ ($C_p = C_{L_p}$ and $R_p = R_{L_p}$).

Note that $R_p$ and $C_p$ are continuously included in $L_0(\M)^{\Nb}$. Therefore, the spaces $R_p + C_p$ and $R_p \cap C_p$ are well-defined. They both contain $S(\M)$ as a dense subspace (weak-$*$ dense for $p=\infty$). 

\begin{remark}\label{rem:HrinGr}
Let $p,q\in (0,\infty]$, the couple $(\Hs_p, \Hs_q)$ is a compatible couple of Banach spaces in the sense of interpolation. Indeed, for all $r>0$, $\Hs_r$ can be identified with a closed subspace of $L_r(\M_{\A})$ by extending the map $G$. Similarly, $R_r$ and $C_r$ are identified with closed subspaces of $L_r(\M_{\B})$, making the couples $(R_p,R_q)$ and $(C_p,C_q)$ compatible.
\end{remark}

With this notations, the first noncommutative Khintchine inequalities state that, assuming that $\xi$ is a sequence of independent Rademacher variables, 
$$
\Hs_p = \left\{
    \begin{array}{ll}
        R_p + C_p  & \mbox{if } p \in [1,2] \\
        R_p \cap C_p & \mbox{if } p \in [2,\infty) \\
    \end{array}
\right.
$$
with equivalent norms.

The main results of sections \ref{section:projection} and \ref{Section alpha} come from the study of optimal decompositions. Let us give their definition right away.
\begin{definition}
For any $x\in S(\M)$ and $p \in (0,1]$. Define 
$$m_{p}(x) = \inf\{\Norm{Ry}_p^p + \Norm{Cz}_p^p : x = z +y, y,z\in S(\M)\}.$$ 
We say that $y,z\in S(\M)$ is an optimal decomposition of $x$ in $L_p$ if $y+z = x$ and $\Norm{Ry}_p^p + \Norm{Cz}_p^p = m_p(x)$.
\end{definition}

The intuition behind this definition is that distributions of optimal decompositions of $x$ should somehow approach the distribution of $Gx$. Though very vague, this idea is partially confirmed by the theorems stated in the next subsection. 

\subsection{Overview of the results}

Recall that $(\M,\tau)$ is a noncommutative integration space and consider a (possibly) noncommutative probability space $\A$ and a sequence $(\xi_i)_{i\geq 0} \in \A^{\Nb}$. To facilitate the reading of this section, we only use standard notations from the litterature. 

For two quantities $A(x)$ and $B(x)$, we will write $A(x) \les B(x)$ or $A(x) \les_C B(x)$, if there is a universal constant $C$ such that for all $x$, $A(x) \leq CB(x)$ and similarly $A(x) \approx B(x)$ or $A(x) \approx_C B(x)$ if there exists $C$ such that for all $x$, $\frac{1}{C}B(x) \leq A(x) \leq CB(x)$.

Our first result is a negative one, we exhibit counterexamples to prove the following proposition:

\begin{property} \label{prop:counterexample}
Suppose that the $\xi_i$ are free Haar unitaries or Rademacher variables and that $\M = B(\ell^2)$.
There is no constant $c$ such that for every finite sequence $x = (x_n)_{n\geq 0} \in \M^{\Nb}$:
$$\Norm{\Sum{i\geq 0}{} x_i \otimes \xi_i}_{2,\infty} \leq c \Norm{x}_{R_{2,\infty}+C_{2,\infty}} . $$
Similarly, there is no constant $c$ such that for every finite sequence $x = (x_n)_{n\geq 0} \in \M^{\Nb}$:
$$\Norm{x}_{R_{2,\infty} \cap C_{2,\infty}} \leq c \Norm{\Sum{i\geq 0}{} x_i \otimes \xi_i}_{2,\infty}. $$
\end{property} 

The counterexamples are constructed using the Schur-Horn theorem. As mentioned in remark \ref{rem:counterexamplegeneral}, the method can be used to prove more general results (see \cite{Cad18Maj}) .

The next theorem has the advantage of requiring very little conditions on the variables considered. We prove that if the lower Khintchine inequality holds for some $p$, it also holds for $q>p$ with the same decomposition.

\begin{theorem} \label{thm:intro:lower}
Let $p \in (0,1]$. Suppose that:
\begin{itemize}
\item the $\xi_i$ are orthonormal in $L_2(\A)$,
\item the $\xi_i$ verify the lower bound of the Khintchine inequality in $L_p$ i.e for any finite sequence $x = (x_i)_{i\geq 0} \in \M_c^{\Nb}$:
$$\Norm{\Sum{i\geq 0}{} x_i \otimes \xi_i}_p \gtrsim_A \Norm{x}_{R_p +C_p}. $$
\end{itemize}
Then there is a decomposition $y,z$ such that $x = y + z$ and for all interpolation space $E$ between $L_p$ and $L_{\infty}$:
$$ \Norm{\p{\Sum{i\geq 0}{}y_iy_i^*}^{1/2}}_{E(\M)} + \Norm{\p{\Sum{i\geq 0}{} z_i^*z_i}^{1/2}}_{E(\M)} \les_c \Norm{\Sum{i\geq 0}{} x_i \otimes \xi_i}_{E(\M_{\A})}. $$
The constant $c$ only depends on $A$ and $p$. 
\end{theorem}

The proof is a dual version of the argument used in \cite{RicardDirksen}. We use the Khintchine inequality in $L_p$ not only for an element $x\in S(\M)$ but also for elements of the form $ex$ where $e$ is a well chosen projection in $\M$. This enables us to obtain a control in terms of $K$-functionals which, by proposition \ref{prop:kfuncandinterpolation}, immediatly implies the theorem above as a corollary. For this idea to work the decomposition $y,z$ has to be close enough to an optimal decomposition.

\begin{remark} \label{rem:Kh(p)_indep_of_p}
The second condition on $\xi$ is actually independent of $p$ for $p<2$. This is a consequence of \cite{PisierRicard2017} (to go from $p$ to $q<p$) and theorem \ref{Main_1} (to go from $p$ to $q>p$). This means that applying the theorem above we can prove Khintchine inequalities in $L_{1,\infty}$ for any sequence of variables $\xi_i$ that verifies Khintchine inequalities in $L_p$ for some $p<2$. More details are given in Corollary \ref{cor:weak1khintchine}. 

Note also that by remark \ref{rem:1<p<2}, theorem \ref{thm:intro:lower} also holds for $p\in (1,2)$.
\end{remark}

The second theorem summarizes the results found in section $\ref{Section alpha}$. By pushing further the properties of optimal decompositions in $L_1$ and together with the key inequality due to Ricard (\cite{Ric18}) we obtain a new proof of Khintchine inequalities in all interpolation spaces between $L_p$-spaces if the variables are for example free unitaries. We also have a "deterministic" equivalent of $\Norm{\Sum{i\geq 0}{} x_i \otimes \xi_i}_{2,\infty}$ which is however less explicit than the usual Khintchine inequalities.  

\begin{theorem} \label{thm:intro:alpha}
Let $x = (x_i)_{i\geq 0} \in S(\M)$. There exist $\alpha,\beta \in \M_c^+$ and $u = (u_i)_{i\geq 0} \in S(\M)$ such that:
\begin{itemize}
\item $s(\alpha) \leq \Sum{i\geq 0}{} u_iu_i^* \leq 1$,
\item $s(\beta) \leq \Sum{i\geq 0}{} u_i^*u_i \leq 1$,
\item for all $i\geq 0$, $x_i = u_i \beta + \alpha u_i$.
\end{itemize}
Furthermore, suppose that:
\begin{itemize}
\item the $\xi_i$ are orthonormal in $L_2(\A)$
\item the $\xi_i$ verify the Khintchine inequality in $L_{\infty}$ i.e for any finite sequence $x = (x_i)_{i\geq 0}\in \M^{\Nb}$:
\[
\Norm{x}_{R_{\infty}\cap C_{\infty}} \approx_{c_{\infty}} \Norm{\Sum{i\geq 0}{} x_i \otimes \xi_i}_{\infty}. \addtag\label{eq:Khinfty}
\]
\end{itemize}
Then for all $p\in (0,\infty)$ and $E$ an interpolation space between $L_p$ and $L_{\infty}$:
$$\Norm{\alpha}_{E(\M)} + \Norm{\beta}_{E(\M)} \approx_c \Norm{\Sum{i\geq 0}{} x_i \otimes \xi_i}_{E(\M_{\A})},$$
where $c$ only depends on $p$ and the constant $c_{\infty}$ appearing in \eqref{eq:Khinfty}.

\end{theorem}

Apart from the inequality found in \cite{Ric18}, the proof is based on a short duality argument.

\subsection{More on the $K$-functional} \label{section:kfunc}

In the context of noncommutative integration, the $K$-functional does not depend, up to universal constants, on the von Neumann algebra in which it is calculated. This result was proved by Xu in his unpublished lecture notes (\cite{Xu_notes}). We will give an alternative proof here using the following proposition which we will also need in the remainder of the paper. It is a version of the power theorem (see \cite{BerghLofstrom}) in the particular case of noncommutative $L_p$-spaces.

\begin{property}\label{prop:powertheorem}
Let $\alpha > 0$, $x\in (L_p(\M)+L_q(\M))^+$ and $p,q\in (0,\infty]$ then:
$$K_t(x,L_p(\M),L_q(\M)) \leq c_{p,\alpha}\bbox{K_{t^{1/\alpha}}(x^{1/\alpha},L_{p\alpha}(\M),L_{q\alpha}(\M))}^{\alpha}, $$
where $c_{p,\alpha} = \max (1,2^{1/p\alpha - 1})\max (2^{\alpha-1},2^{1 - \alpha})$.
\end{property}

\begin{remark}
We will use this proposition under the following form. Let $\alpha > 0$, $x\in (L_p(\M)+L_q(\M))^+$ and $p,q\in (0,\infty]$ then:
$$K_t(x,L_p(\M),L_q(\M)) \approx \bbox{K_{t^{1/\alpha}}(x^{1/\alpha},L_{p\alpha}(\M),L_{q\alpha}(\M))}^{\alpha}, $$
\end{remark}

We use the following routine operator inequalities, proved in the last section of the paper.

\begin{lemma} \label{lem:matrixineq}
Let $a,b\in L_0(\M)^+$, $\alpha \geq 1$ and $\theta \leq 1$. 
\begin{enumerate}[label=\text{\roman*.}]
\item if $0 \leq a \leq b$ then there exists a contraction $c$ such that $a = cbc^*$,
\item if $0 \leq a \leq b$ then there exists a partial isometry $u$ such that $a^2 \leq ub^2u^*$,
\item there exists a partial isometry $u\in\M$ such that:
$$ (a + b)^{\alpha} \leq 2^{\alpha-1}u(a^{\alpha} + b^{\alpha})u^*, $$
\item there exist two partial unitaries $u$ and $v$ such that:
$$(a+b)^{\theta} \leq ua^{\theta}u^* + vb^{\theta}v^*.$$
\end{enumerate}
\end{lemma}

\begin{proof}[Proof of proposition \ref{prop:powertheorem}]
For all $r>0$, define $a_r := \max (1,2^{1/r - 1})$. Take $a\in L_{p\alpha}(\M)$ and $b\in L_{q\alpha}(\M)$ such that $a + b = x^{1/\alpha}$. We start by considering the case $\alpha <1$. We need to find $a'\in L_{p}$ and $b'\in L_{q}$ such that $a' + b' = x$ and:
\[
\Norm{a'}_{p} + t\Norm{b'}_{q} \leq a_{p\alpha}2^{1-\alpha}\p{\Norm{a}_{p\alpha} + t^{1/\alpha}\Norm{b}_{q\alpha}}^{\alpha}.
\]
Since $x$ is positive, we can suppose that $a$ and $b$ are positive. Indeed, first note that:
$$x = \dfrac{a + a^*}{2} + \dfrac{b + b^*}{2},$$
and:
$$\Norm{\dfrac{a + a^*}{2}}_{p\alpha} + t\Norm{\dfrac{b + b^*}{2}}_ {q\alpha} \leq a_{p\alpha} (\Norm{a}_{p\alpha} + t\Norm{b}_{q\alpha})$$
using the triangular inequality for $p\alpha \geq 1$ and the $p\alpha$-triangular inequality for $p\alpha \leq 1$. So we can suppose that $a$ and $b$ are selfadjoint and write $a = a_+ - a_-$ and $b = b_+ - b_-$ their decompositions into positive and negative parts. It follows that $x\leq a_+ + b_+$ so there exists a contraction $c$ such that $x = ca_+c^* + cb_+c^*$ which yields a better decomposition than $a,b$.

Using lemma $\ref{lem:matrixineq}$, we can find two contractions $u$ and $v$ such that:
$$x = ua^{\alpha}u^* + vb^{\alpha}v^* =: a' + b'.$$
And we have:
\begin{align*}
\Norm{a'}_{p} + t\Norm{b'}_{q} &\leq \Norm{a^{\alpha}}_p + \Norm{b^{\alpha}}_q \\
&= \Norm{a}_{p\alpha}^{\alpha} + t^{\alpha}\Norm{b}_{q\alpha}^{\alpha} \\
&\leq 2^{1-\alpha}\p{\Norm{a}_{p\alpha} + t\Norm{b}_{q\alpha}}^{\alpha}. 
\end{align*}
Taking the infimum over all decompositions $x^{1/\alpha} = a + b$, and recalling that we had to add a factor $a_{p\alpha}$ to consider only $a$ and $b$ positive, we obtain:
$$K_t(x,L_p(\M),L_q(\M)) \leq a_{p\alpha}2^{1-\alpha} \bbox{K_{t^{1/\alpha}}(x^{1/\alpha},L_{p\alpha}(\M),L_{q\alpha}(\M))}^{\alpha}.$$
Let us now consider $\alpha \geq 1$. Continue to assume $a$ and $b$ to be nonnegative and recall that we lose a constant $a_{p\alpha}$ by doing so. Using lemma $\ref{lem:matrixineq}$ there exists a contraction $u$ such that:
$$ x = 2^{\alpha - 1}u(a^{\alpha} + b^{\alpha})u^* = 2^{\alpha - 1}ua^{\alpha}u^* + 2^{\alpha - 1}ub^{\alpha}u^* =: a' + b'.$$
Now, we compute:
\begin{align*}
\Norm{a'}_p + t\Norm{b'}_q &\leq 2^{\alpha - 1}(\Norm{a}_{p\alpha}^{\alpha} + t\Norm{b}_{q\alpha}^{\alpha}) \\
&\leq 2^{\alpha - 1}(\Norm{a}_{p\alpha} + t^{1/\alpha}\Norm{b}_{q\alpha})^{\alpha}.
\end{align*}
Taking the infimum over all decompositions, the proof is complete.
\end{proof}

\begin{property} \label{prop:Kt_independant of M}
Let $x\in L_p(\M) + L_q(\M)$, $p,q\in\Rb$ such that $0<p<q$ and $t>0$:
$$K_t(\mu(x),L_p(\Rb^+),L_q(\Rb^+)) \approx_{c_p} K_t(x,L_p(\M),L_q(\M)),$$ 
where $c_p$ only depends on $p$.
\end{property}

\begin{proof}[Proof of proposition \ref{prop:Kt_independant of M}]
We will rely on lemmas \ref{lem:K_t approx finie} and \ref{lem:tech_approx_mu_finite} whose proofs can be found in the last section of the paper. 

Suppose that $x$ is positive. We do not lose generality here since multiplying by a unitary does not change the $K$-functional. 

First, assume that $p>1$. Let $e\in \Pc_c(\M)$ be a finite projection. Let us prove the proposition for $y = exe$. To do so, we can work in the finite algebra $e\M e$ which clearly does not modify the $K$-functional of $y$. Denote by $\M_{y}$ the von Neumann algebra generated by $y\in\M$ which is abelian. There are two conditional expectations $E_1 : \M \to \M_{y}$ and $E_2 : L_{\infty}(0,\tau(e)) \to \M_{\mu(y)}$ and $\M_{y}$ is canonically isomorphic to $\M_{\mu(y)}$ by $y \mapsto \mu(y)$. Since conditional expectations extend to contractions on $L_p$, we have, 
\[
K_t(y,L_p(\M),L_q(\M)) = K_t(\mu(y),L_p(\Rb^+),L_q(\Rb^+)). 
\]

Let us consider $x$ once again. Combining the equality above with lemma \ref{lem:K_t approx finie}, we obtain:
\[
K_t(x,L_p(\M),L_q(\M)) = \supb{e\in\Pc_c(\M)} K_t(\mu(exe),L_p(0,\infty),L_q(0,\infty)) \addtag \label{eq:Kt_Mtosupmu}
\]
and 
\[
K_t(\mu(x),L_p(0,\infty),L_q(0,\infty)) = \supb{f\in\Pc_c(L_{\infty}(0,\infty))} K_t(\mu(f\mu(x)f),L_p(0,\infty),L_q(0,\infty)).
\addtag \label{eq:Kt_mutosupmu}
\]
Since for all $e\in \Pc(\M)$, $\mu(exe) \leq \mu(x)$ (see \cite{FackKosaki}): 
$$K_t(x,L_p(\M),L_q(\M)) \leq K_t(\mu(x),L_p(0,\infty),L_q(0,\infty)).$$
Conversely, by lemma \ref{lem:tech_approx_mu_finite}, for every projection $f\in \Pc_c(L_{\infty}(0,\infty))$ and $\e>0$, there exists a finite projection $e$ in $\M$, such that $\mu(exe) + \e \geq \mu(f\mu(x))$. This implies that
$$\supb{e\in\Pc_c(\M)} K_t(\mu(exe),L_p(0,\infty),L_q(0,\infty)) \geq \supb{f\in\Pc_c(L_{\infty}(0,\infty))} K_t(f\mu(x)f,L_p(0,\infty),L_q(0,\infty)),$$
and enables us to conclude using \eqref{eq:Kt_Mtosupmu} and \eqref{eq:Kt_mutosupmu}.

Now take $0<p \leq 1$ and $q>p$. The result follows from proposition $\ref{prop:powertheorem}$:
\begin{align*}
K_t(y,L_p(\M),L_q(\M)) &\approx \bbox{K_{t^{p/2}}(y^{p/2},L_2(\M),L_{2q/p}(\M)}^{2/p}\\
&= K_{t^{p/2}}(\mu(y^{p/2}),L_2(\Rb^+),L_{2q/p}(\Rb^+))^{2/p} \\
&\approx K_t(\mu(y),L_p(\Rb^+),L_q(\Rb^+)).
\end{align*}

\end{proof}

\begin{remark}
In \cite{Xu_notes}, the result above is obtained by first proving that the couple $(L_1(\M),L_{\infty}(\M))$ is, in interpolation language, a partial retract of the couple $(L_1(0,\infty),L_{\infty}(0,\infty))$.
\end{remark}

This enables us to define: $K_t(x,p,q) := K_t(\mu(x),L_p(\Rb^+),L_q(\Rb^+))$.

\begin{remark}\label{rem:KtRowCol}
The equivalence above also applies to the row and column spaces. More precisely, for $x\in S(\M)$ and $p,q >0$,
\begin{center}
$K_t(x,R_p,R_q) \approx K_t(Rx,p,q)$ and $K_t(x,C_p,C_q) \approx K_t(Cx,p,q)$. 
\end{center}
\end{remark}

\begin{proof}
Using the fact that for all $r>0$, $R_r$ and $C_r$ are complemented in $L_r(\M\overline{\otimes}\B(\ell^2))$ and proposition \ref{prop:Kt_independant of M},
$$K_t(x,R_p,R_q) \approx K_t(Rx,L_p(\M\overline{\otimes}\B(\ell^2)),L_q(\M\overline{\otimes}\B(\ell^2))) \approx K_t(Rx,p,q)$$
and similarly for columns.

\end{proof}

Recall the following formula for the particular case $q=\infty$ (see \citep{BerghLofstrom}). For all $p>0$, there exists $A_p \in \Rb^+$ (with $A_1 = 1$) such that:
\[
\p{\Int{0}{t^p} \mu_s(x)^p ds}^{1/p} \approx_{A_p} K_t(x,p,\infty), \addtag \label{eq:formuleKt}
\]
From this, we deduce the following expression for $K_t(x,p,\infty)$ in terms of a supremum over projections, which is the important result of this section.

\begin{property} \label{prop:approx:eKt}
Suppose that $\M$ is diffuse. Let $p>0$. Then for all $x\in\M_c$:
$$\sup \{ \Norm{ex}_p : \tau(e) \leq t^p, e\in \Pc(\M) \} \approx_{A_p} K_t(x,p,\infty).$$
\end{property}

\begin{proof}
If $t^p > \tau(1)$, we have:
$$\sup \{ \Norm{ex}_p : \tau(e) \leq t^p, e\in \Pc(\M) \} = \Norm{x}_p$$
and the proposition is verified by formula (\ref{eq:formuleKt}). So from now on, assume that $t^p\leq \tau(1)$.

Since $\M$ is diffuse, there is a projection $e$ in $\M$ with trace $t^p$, commuting with $\md{x^*}$, such that: 
$$\tau(e\md{x^*}^p) = \Int{0}{t^p} \mu_s(x)^p ds \geq \frac{1}{A_p} K_t(x,p,\infty)^p,$$
where we used \eqref{eq:formuleKt} to obtain the inequality. Furthermore, since $e$ and $\md{x^*}$ live in a commutative von Neumann algebra, they can be represented as functions in a space $L_{\infty}(\Omega)$, thus: 
$$\tau(e\md{x^*}^p) = \tau(e^{p/2}(xx^*)^{p/2}e^{p/2}) = \tau((exx^*e)^{p/2}) = \Norm{ex}_p^p.$$
Hence:
$$ K_t(x,p,\infty) \leq A_p \Norm{ex}_p \leq A_p \sup \{ \Norm{ex}_p : \tau(e) = t^p, e\in \Pc(\M) \}.$$
To prove the converse inequality, take $e \in \Pc(\M)$ such that $\tau(e) = t^p$.
Note that by \eqref{eq:formuleKt}:
$$\Norm{ex}_p^p = \Int{0}{\infty} \mu(\md{ex}^p) = \Int{0}{t^p} \mu(\md{ex})^p \leq A_p K_t(ex,p,\infty)^p.$$
Furthermore for all $s \in \Rb^+$, $\mu_s(ex) \leq \mu_s(x)$, see for example \cite{FackKosaki}. Hence, $K_t(ex,p,\infty) \leq K_t(x,p,\infty)$. Combining the two previous inequalities, we obtain:
$$\Norm{ex}_p \leq A_p K_t(ex,p,\infty) \leq A_p K_{t}(x,p,\infty).$$

\end{proof}

\begin{remark} \label{remark:approxeKt:commutation}
 The proof yields a bit more than the proposition. Indeed, it suffices to consider the supremum over projections $e$ commuting with $\md{x^*}$ to obtain the left hand side inequality. This will be of importance later on.
\end{remark}

\section{Some properties of optimal decompositions in $L_p$ for $p < 2$} \label{section:projection}

\subsection{Main result and consequences}

We stick with the notations introduced in $\S\ref{section:noncommutative khintchine inequalities}$ . In this section, the variables $\xi_i$ will always satisfy the following conditions:

\begin{enumerate}
\item the $\xi_i$ are orthonormal in $L_2(\A)$,
\item the $\xi_i$ verify the lower Khintchine inequality for some $p \leq 1$. More precisely, there exists a constant $B_p$ such that for all $x\in S(\M)$, 
\[
m_p(x)^{1/p} \leq B_p \Norm{Gx}_p. \addtag \label{Kh(p)}
\]
\end{enumerate}

Typical examples of such variables include free Haar unitaries or Rademacher variables. We focus on lower Khintchine inequalities i.e of the type:
$$\Norm{x}_{R_E + C_E} \les \Norm{Gx}_E,$$
for $E$ a symmetric space and $x\in S(\M)$. The converse inequality presents no difficulty in the motivating example of $L_{1,\infty}$ as we will see later. In \cite{RicardDirksen}, it is shown that by applying multiple times Khintchine inequality in $L_{\infty}$ for an element $x$, one can obtain a majoration of $\mu(Gx)$. Though it is less direct, our method is similar and by using the Khintchine inequality in $L_p$ for $p<2$ we obtain a minoration of the $K$-functional of $Gx$. The main theorem of this section is the following.

\begin{theorem} \label{Main_1}
Let $p\in (0,1]$ such that \eqref{Kh(p)} holds. There exists a constant $C_p$ such that for all $x\in S(\M)$ there exist $y,z\in S(\M)$ such that $y + z = x$ and for all $t\geq 0$, 
\begin{center}
$K_t(Ry, p, \infty) \leq C_p K_t(Gx, p, \infty)$ and $K_t(Cz, p, \infty) \leq C_p K_t(Gx, p, \infty)$.
\end{center}
The constant $C_p$ only depends on $A_p$ and $B_p$ which appear in \eqref{eq:formuleKt} and \eqref{Kh(p)} respectively.
\end{theorem} 

Let us highlight some consequences of the theorem. Since we have a control on the $K$-functional, it extends to all interpolation spaces between $L_p$-spaces by proposition \ref{prop:kfuncandinterpolation} and thus we obtain theorem \ref{thm:intro:lower} mentioned in the preliminaries as a corollary. It also allows us to prove the Khintchine inequalities in $L_{1,\infty}$.

\begin{corollary}[The Khintchine inequality in $L_{1,\infty}$] \label{cor:weak1khintchine}
For any $x\in S(\M)$,
$$\Norm{Gx}_{1,\infty} \approx \Norm{x}_{R_{1,\infty} + C_{1,\infty}}.$$
\end{corollary}

\begin{proof}
By remark \ref{rem:Kh(p)_indep_of_p}, we can suppose that $p < 1$. The inequality:
$$ \inf \{\Norm{Ry}_{1,\infty} + \Norm{Cz}_{1,\infty} : x = y + z, y,z\in S(\M)\} \les \Norm{Gx}_{1,\infty} $$
is given by theorem \ref{thm:intro:lower} applied to $E = L_{1,\infty}$.
The converse inequality is classical. We know that the map $Ry \mapsto Gy$ is a contraction on $L_{p}$ and $L_{2}$, hence, by interpolation, it is bounded on $L_{1,\infty}$. By adjunction, the same is true for $Cz \mapsto Gz$. Hence:
$$\Norm{Gx}_{1,\infty} \les \Norm{Gy}_{1,\infty} + \Norm{Gz}_{1,\infty} \les \Norm{Ry}_{1,\infty} + \Norm{Cz}_{1,\infty}.$$
\end{proof}

\begin{remark}
We emphasized the previous corollary because it concerns a special case that motivated our work but the exact same proof works for a general interpolation space betwenn $L_p$, $p<2$ and $L_2$ using interpolation for the upper bound and theorem \ref{thm:intro:lower} for the lower bound. More precisely, let $p<2$ and $E$ an interpolation space between $p$ and $2$ then:
$$\Hs_E = R_E + C_E$$
with equivalent quasi-norms.
\end{remark}

\begin{remark}
As a consequence of the preceding remark and theorem \ref{Main_1}, we obtain the following.
Let $p,q \in (0,2]$, $p\leq q$. Then for all $x\in S(\M)$,
$$K_t(x,\Hs_p,\Hs_q) \les K_t(Gx,p,q).$$
Using the terminology introduced in \cite{Pis92}, the couple $(\Hs_p,\Hs_q)$ is $K$-closed in $(L_p(\M_{\A}), L_q(\M_{\A}))$ (recall that $\Hs_p$ and $\Hs_q$ are identified with subspaces of $L_p(\M_{\A})$ and $L_q(\M_{\A})$ by remark \ref{rem:HrinGr}). This means that the $\Hs_p$-spaces behave well with respect to interpolation for $p\leq 2$.
\end{remark}

\begin{proof}
Let $t>0$ and note that $L_p + tL_q$ is an interpolation space between $L_p$ and $L_q$. Using the fact that $L_q$ is an interpolation space between $L_p$ and $L_{\infty}$ and by applying the definition of an interpolation space, this means that $L_p + tL_q$ is an interpolation space between $L_p$ and $L_{\infty}$. Let $x\in S(\M)$. This means by theorem \ref{Main_1} that there exists $y,z\in S(\M)$ such that $x = y+z$ and:
$$K_t(Ry,p,q) + K_t(Cz,p,q) \les K_t(Gx,p,q).$$
Hence, by remark \ref{rem:KtRowCol}, we have:
$$K_t(y,R_p,R_q) + K_t(z,C_p,C_q) \les K_t(Gx,p,q).$$
Let $y_1,y_2,z_1,z_2 \in S(\M)$ such that $y = y_1 + y_2$, $z = z_1 + z_2$,
\begin{center}
$\Norm{y_1}_{R_p} + t\Norm{y_2}_{R_q} \les K_t(y,R_p,R_q)$ and $\Norm{z_1}_{R_p} + t\Norm{z_2}_{R_q} \les K_t(z,C_p,C_q)$
\end{center}
Let $x_1 = y_1 + z_1$ and $x_2 = y_2 + z_2$, combining the previous inequalities:
$$\Norm{x_1}_{R_p + C_p} + t\Norm{x_2}_{C_p+C_q} \les K_t(Gx,p,q).$$
Since, by the previous remark $R_p + C_p \approx \Hs_p$ and $R_q + C_q \approx \Hs_q$, we obtain:
$$K_t(x,\Hs_p,\Hs_q) \les K_t(Gx,p,q).$$
\end{proof}

\subsection{First steps towards the proof}

In this part, we present the main ideas that will allow us to prove theorem \ref{Main_1}. The central one is contained in proposition \ref{prop:almostmain}. Starting with an element $x$, we use the Khintchine inequality on $e(Gx)$ for well chosen projections $e \in \Pc (\M)$ and thanks to proposition \ref{prop:approx:eKt} we deduce the expected control on $K$-functionals. There are, however, two technical difficulties. The first one is that we could not prove that an optimal decomposition always exists when $p<1$. To skirt this problem, in the next part, we will prove that the argument also works for decompositions that are close enough to being optimal but in this case we need an additional control on the operator norm of the decomposition which is given by lemma \ref{control:infty}. The second difficulty is that to use propositon \ref{prop:approx:eKt}, we need to work in a diffuse algebra. To that effect, we simply tensor our base algebra $\M$ by $L_{\infty}(0,1)$ which fixes the proof immediatly thanks to lemma \ref{lem:eq:diffuse}.

\begin{property} \label{prop:almostmain}
Suppose that $\M$ is diffuse. Let $p\in (0,1]$, if $x\in S(\M)$ and $x = y+z$ is an optimal decomposition in $L_p$ then for all $t\geq 0$, $K_t(Ry, p, \infty) \leq A_p^2 B_p K_t(Gx, p, \infty)$ and $K_t(Cz, p, \infty) \leq A_p^2 B_p K_t(Gx, p, \infty)$. 
\end{property}

\begin{proof}
Let $e$ be a projection commuting with $\md{(Ry)^*}$ and $f = 1-e$. Take $\e > 0$ and $y_1,y_2,z_1,z_2\in S(\M)$ such that $ex = y_1 + z_1$, $fx = y_2 + z_2$, $\Norm{Ry_1}_p^p + \Norm{Cz_1}_p^p \leq m_{p}(ex) + \e$ and $\Norm{Ry_2}_p^p + \Norm{Cz_2}_p^p \leq m_{p}(fx) + \e$. We can write $x = ex + fx = y_1 + z_1 + y_2 + z_2$. Then by minimality of $y$ and $z$:
$$\Norm{Ry}_p^p + \Norm{Cz}_p^p \leq \Norm{R(y_1 + y_2)}_p^p + \Norm{C(z_1 + z_2)}_p^p.$$ 
We use the $p$-triangular inequality and obtain that:
$$\Norm{R(y_1 + y_2)}_p^p + \Norm{C(z_1 + z_2)}_p^p \leq \Norm{Ry_1}_p^p + \Norm{Cz_1}_p^p + \Norm{Ry_2}_p^p + \Norm{Cz_2}_p^p.$$
Combined with lemma $\ref{lem:trick}$, we get:
$$\Norm{R(ey)}_p^p + \Norm{R(fy)}_p^p + \Norm{Cz}_p^p \leq \Norm{Ry_1}_p^p + \Norm{Cz_1}_p^p + \Norm{Ry_2}_p^p + \Norm{Cz_2}_p^p.$$
Now using the almost minimality of the couple $(y_2, z_2)$ and lemma $\ref{lem:trick}$, we obtain:
$$\Norm{Ry_2}_p^p + \Norm{Cz_2}_p^p \leq \Norm{R(fy)}_p^p + \Norm{C(fz)}_p^p + \e \leq \Norm{R(fy)}_p^p + \Norm{Cz}_p^p + \e.$$
Hence, combining the last two inequalities and then using the Khintchine inequality for $p$:
$$\Norm{R(ey)}_p^p \leq \Norm{Ry_1}_p^p + \Norm{Cz_1}_p^p + \e \leq B_p^p\Norm{G(ex)}_p^p + 2\e.$$
This is true for all $\e>0$ so:
$$\Norm{R(ey)}_p^p \leq B_p^p\Norm{G(ex)}_p^p.$$
The previous inequality holds for all projections $e$ commuting with $\md{(Ry)^*}$.
Taking the supremum over all such $e$ with $\tau(e) \leq t^p$ and using proposition $\ref{prop:approx:eKt}$ and remark $\ref{remark:approxeKt:commutation}$, we obtain:
$$K_t(Ry, p, \infty) \leq A_p^2 B_p K_t(Gx, p, \infty).$$
The case of $z$ is exactly symmetrical by taking adjoints and so the proof is complete.
\end{proof} 

To prove the theorem without making any assumptions, the following lemma is crucial.

\begin{lemma} \label{control:infty}
Let $\e > 0$ and $x\in S(\M)$. There exist $y$ and $z$ in $S(\M)$ such that $x = y + z$, $\Norm{Ry}_p^p + \Norm{Cz}_p^p \leq m_p(x) + \e$, $\Norm{Ry}_{\infty} \leq 2\Norm{Gx}_{\infty}$ and $\Norm{Cz}_{\infty} \leq 2\Norm{Gx}_{\infty}$.
\end{lemma}

\begin{proof}
With the notations of the lemma, choose $y,z\in S(\M)$ such that $y + z = x$ and $\Norm{Ry}_p^p + \Norm{Cz}_p^p \leq m_p(x) + \e$ and denote $\Norm{Gx}_{\infty} = A$. Let $e = \Ind_{[A,\infty)}(\md{(Ry)^*})$, $f = \Ind_{[A,\infty)}(\md{Cz})$. We write $x = e^{\bot}xf^{\bot} + ex + e^{\bot}xf$ and deduce the new decomposition: $y' = e^{\bot}yf^{\bot} + ex$ and $z' = e^{\bot}zf^{\bot} + e^{\bot}xf$. Let us check that it satisfies the conditions of the lemma.

Note that $\Norm{R(ex)}_{\infty} \leq \Norm{R(x)}_{\infty} \leq A$. The last inequality follows from the fact that the $\xi_i$ are orthonormal and that consequently the map $Id \otimes \tau_{\A}: \M \overline{\otimes} \A \to \M$ sends $\md{(Gx)^*}^2$ to $\md{(Rx)^*}^2$. Moreover, the left support of $R(ex)$ is less than $e$, indeed: $\md{(R(ex))^*}^2 = e(\Sum{i}{} x_ix_i^*)e$. Hence $\md{R(ex)^*} \leq Ae$. Note also that $\md{(R(ey))^*} \geq Ae$. Indeed, since $e$ and $\md{(Ry)^*}$ commute by lemma $\ref{lem:trick}$, $\md{(R(ey))^*} = e\md{(Ry)^*} \geq Ae$ by definition of $e$. By symmetry, we have the same kind of estimates for the columns i.e $\md{C(e^{\bot}xf)} \leq Af$ and $\md{C(zf)} \geq Af$. For rows, we get: 
\begin{align*}
\Norm{Ry'}_p^p &\leq \Norm{R(e^{\bot}yf^{\bot})}_p^p + \Norm{R(ex)}_p^p \\
&\leq \Norm{R(e^{\bot}y)}_p^p + \tau(e)A^p \\
&\leq \Norm{R(e^{\bot}y)}_p^p + \Norm{R(ey)}_p^p \\
&= \Norm{Ry}_p^p 
\end{align*}
where the last inequality is given by lemma $\ref{lem:trick}$.
Similarly, for columns, we get:
$$\Norm{Cz'}_p^p \leq \Norm{Cz}_p^p.$$
Consequently:
$$\Norm{Ry'}_p^p + \Norm{Cz'}_p^p \leq \Norm{Ry}_p^p + \Norm{Cz}_p^p \leq m_p(x) + \e.$$
The control in $L_{\infty}$ also follows quickly:
\begin{align*}
\Norm{Ry'}_{\infty} &\leq \Norm{R(e^{\bot}yf^{\bot})}_{\infty} + \Norm{R(ex)}_\infty \\
&\leq \Norm{R(e^{\bot}y)}_{\infty} + \Norm{Rx}_{\infty} \\
&\leq 2A
\end{align*}
where the last inequality is a consequence of the definition of $e$.
The case of columns is as always symmetrical which concludes the proof.
\end{proof}

The following lemma is the key to remove the hypothesis that $\M$ is diffuse.

\begin{lemma}\label{lem:eq:diffuse}
Consider the noncommutative measure space $\N = \M \overline{\otimes} L_{\infty}([0,1])$ equipped with the tensor product trace and identify $\M$ with $\M \otimes 1 \subset \N$. Then for any $p>0$ and $x\in S(\M)$:
$$ m_p(x) = \inf \{ \Norm{Rf}_p^p + \Norm{Cg}_p^p : f + g = x, f,g\in S(\N) \}. $$
\end{lemma}

\begin{proof}
Since $\M \subset \N$ the inequality:
$$ m_p(x) \geq \inf \{ \Norm{Rf}_p^p + \Norm{Cg}_p^p : f + g = x, f,g\in S(\N) \}$$
is clear.

Let $f,g \in S(\N)$ such that $f + g = x$. Seeing $f$ and $g$ as functions from $[0,1]$ to $S(\M)$ we write:
$$ \Norm{Rf}_p^p + \Norm{Cg}_p^p = \Int{0}{1} \Norm{R(f(t))}_p^p + \Norm{C(g(t))}_p^p dt \geq \Int{0}{1} m_p(x) dt = m_p(x).$$
\end{proof}
Extend the notation $m_p(h)$ to elements $h\in S(\N)$ by:
$$m_p(h) := \inf \{ \Norm{Rf}_p^p + \Norm{Cg}_p^p : f + g = h, f,g\in S(\N) \}.$$
When $x\in S(\M)$ is considered as an element of $S(\N)$, there are now two definitions of $m_p(x)$. Either the decompositions can be taken is $S(\N)$ or in $S(\M)$. Note that by the lemma above, these two definitions coincide. 

\subsection{Proof of the theorem in full generality}

In this section, we present a proof of the main result using decompositions that are close to be optimal rather than optimal. We essentially follow the proof of proposition $\ref{prop:almostmain}$ but we need some additional care and lemma $\ref{control:infty}$ to get the final estimate. Another approach is to work in an ultraproduct where optimal decompositions always exist. The two strategies yield the same constants but we present the elementary one since it ended up also being the less technical.

\begin{lemma} \label{lem:est:Ktb}
Let $p\in (0,1]$. For all $x\in S(\M)$, $\eta>0$ and decompositions $x=y+z$ such that $\Norm{Ry}_p^p + \Norm{Cz}_p^p \leq m_p(x) + \eta^p$, we have:
$$K_t(Ry,p,\infty) \leq A_p\p{A_p^p B_p^pK_t(Gx,p,\infty)^p + \eta^p}^{1/p}$$
for all $t>0$.
\end{lemma}

\begin{proof}
Let $p\in (0,1]$ and  $x\in S(\M)$. Take $\eta > 0$ and $y,z \in S(\M)$ such that $y + z = x$ and $\Norm{Ry}_p^p + \Norm{Cz}_p^p \leq m_p(x) + \eta^p$. To be able to use lemma $\ref{prop:approx:eKt}$, we need to work in a diffuse algebra so we will now consider $x,y$ and $z$ as elements of $S(\N)$.

We can now repeat the argument of the proof of proposition $\ref{prop:almostmain}$. Take $e$ a projection commuting with $\md{(Ry)^*}$ in $S(\N)$ and $f = 1 - e$. Take $\e>0$ and $ex = y_1 + z_1$ , $fx = y_2 + z_2$ such that $\Norm{Ry_1}_p^p + \Norm{Cz_1}_p^p \leq m_{p}(ex) + \e$ and $\Norm{Ry_2}_p^p + \Norm{Cz_2}_p^p \leq m_{p}(fx) + \e$.  By definition of $y$ and $z$, $\Norm{Ry}_p^p + \Norm{Cz}_p^p \leq m_p(x) + \eta^p$, hence :
$$\Norm{Ry}_p^p + \Norm{Cz}_p^p \leq \Norm{R(y_1 + y_2)}_p^p + \Norm{C(z_1 + z_2)}_p^p + \eta^p.$$ 
By the $p$-triangular inequality and lemma $\ref{lem:trick}$:
$$\Norm{R(ey)}_p^p + \Norm{R(fy)}_p^p + \Norm{Cz}_p^p \leq \Norm{Ry_1}_p^p + \Norm{Cz_1}_p^p + \Norm{Ry_2}_p^p + \Norm{Cz_2}_p^p + \eta^p.$$
Now using the almost minimality of the couple $(y_2, z_2)$ and lemma $\ref{lem:trick}$, we obtain:
$$\Norm{Ry_2}_p^p + \Norm{Cz_2}_p^p \leq \Norm{R(fy)}_p^p + \Norm{C(fz)}_p^p + \e \leq \Norm{R(fy)}_p^p + \Norm{Cz}_p^p + \e.$$
And we conclude using the Khintchine inequality: 
$$\Norm{R(ey)}_p^p \leq \Norm{Ry_1}_p^p + \Norm{Cz_1}_p^p + \e + \eta^p \leq B_p^p\Norm{G(ex)}_p^p + 2\e + \eta^p.$$
This is true for all $\e>0$ so:
$$\Norm{R(ey)}_p^p \leq B_p^p\Norm{G(ex)}_p^p + \eta^p.$$

Taking the supremum over all projections $e\in\N$ commuting with $\md{(Ry)^*}$, by proposition $\ref{prop:approx:eKt}$, we obtain that:
\[
K_t(Ry,p,\infty) \leq A_p( A_p^p B_p^p K_t(Gx,p,\infty)^p + \eta^p)^{1/p}. 
\]
\end{proof}

\begin{proof}[Proof of theorem $\ref{Main_1}$]
Let $x\in S(\M)$. We need to find $y$ and $z$ in $S(\M)$ such that $y + z = x$,  
$$K_t(Ry, p,\infty) \les K_t(Gx, p, \infty)$$ 
and similarly for $Cz$. Write again $A = \Norm{Gx}_{\infty}$. Define $\delta := \tau(\Ind_{\md{Gx}>A/2})^{1/p}$. For $t<\delta$, ($\ref{eq:formuleKt}$) gives:
$$A_p K_t(Gx,p,\infty) \geq (\Int{0}{t^p} \mu_u(Gx)^p du)^{1/p} \geq \frac{tA}{2}.$$
Let $ \eta = K_{\delta}(Gx,p,\infty)$.

Take $y,z$ such that $\Norm{Ry}_p^p + \Norm{Cz}_p^p \leq m_p(x) + \eta^p$, $\Norm{Ry}_{\infty} \leq 2A$ and $\Norm{Cz}_{\infty} \leq 2A$. This is possible by lemma $\ref{control:infty}$. For $t\geq \delta$, using lemma $\ref{lem:est:Ktb}$ and since $t \mapsto K_t(Gx,p,\infty)$ is increasing, we have:
\begin{align*}
K_t(Ry,p,\infty) &\leq A_p( A_p^p B_p^p K_t(Gx,p,\infty)^p + \eta^p)^{1/p} \\
&\leq A_p( A_p^p B_p^p K_t(Gx,p,\infty)^p + K_{\delta}(Gx,p,\infty)^p)^{1/p} \\
&\leq A_p(A_p^p B_p^p+1)^{1/p} K_t(Gx,p,\infty).\\
\end{align*}

For $t<\delta$, we know that $A_pK_t(Gx,p,\infty) \geq \dfrac{tA}{2}$ and since $\Norm{Ry}_{\infty} \leq 2A$, we also have $K_t(Ry, p,\infty) \leq 2At$ so $K_t(Ry,p,\infty) \leq 4A_pK_t(Gx,p,\infty)$. 

Let $C_p = \max (A_p(A_p^p B_p^p+1)^{1/p}, 4A_p)$. We have just proven that for all $t>0$, $K_t(Ry,p,\infty) \leq C_pK_t(Gx,p,\infty)$. Since the argument is perfectly symmetrical, we also have $K_t(Cz,p,\infty) \leq C_pK_t(Gx,p,\infty)$.
\end{proof}

\begin{remark}\label{rem:1<p<2}
The theorem still holds when $p\in (1,2]$. Indeed, the $p$-triangular inequality is false in general but here, we only use it to prove inequalities of the type:
\[
\Norm{ex + fy}_p^p \leq \Norm{ex}_p^p + \Norm{fy}_p^p \addtag \label{equa:ptringular}
\]
where $e$ is a projection, $f = 1 - e$ and $x,y\in \M$.
This still holds for $p\in [1,2]$ using the inequality, for all $a,b\in\M$:
\[
\dfrac{\Norm{a-b}_p^p + \Norm{a+b}_p^p}{2} \leq \Norm{a}_p^p + \Norm{b}_p^p. \addtag \label{equa:geo}
\]
Note that $ex + fy = (e-f)(ex - fy)$ and that $e - f$ is a unitary. Hence $\Norm{ex + fy}_p = \Norm{ex - fy}_p$. Now take $a = ex$ and $b = fy$ in $(\ref{equa:geo})$ to obtain $(\ref{equa:ptringular})$.
To prove $(\ref{equa:geo})$ one can for example apply the Riesz-Thorin interpolation theorem to the application $T: (x,y) \mapsto (x+y,x-y)$. 
\end{remark}

\section{Further results on optimal decompositions in $L_1$} \label{Section alpha}

In this section, we investigate further the properties of optimal decompositions in $L_1$. The first notable fact is that in this case, we can prove that an optimal decomposition always exists (see lemma \ref{lem:exist_optimal_decomposition}). Knowing this, a simple duality argument yields a factorisation for elements in $S(\M)$ (theorem \ref{thm:factorization}). Remarkably, this result is of purely algebraic nature and combined with \cite{Ric18} which provides the necessary estimate on anticommutators, produces a new proof of Khintchine inequalities. The main novelty is the emergence of elements $\alpha,\beta \in \M^+$ associated to $x\in S(\M)$ which play the role of a "modulus" in $\Hs_p$-spaces (theorem \ref{thm:intro:alpha}). In particular, there is no need in the proofs to distinguish between $p\leq 2$ or $p\geq 2$. The drawback of the method is that it relies on Khintchine inequalities in $L_{\infty}$ and therefore does not apply, at least directly, to Rademacher variables. Let us already assume that for all $x\in S(\M)$, 
\[ 
\Norm{Gx}_{\infty} \les_{c_{\infty}} \max (\Norm{Rx}_{\infty}, \Norm{Cx}_{\infty}). \addtag \label{equa:ineqKhininfty}
\]
We start by proving the existence of an optimal decomposition in $L_1$. The argument is straightforward by taking a limit of a minimising sequence of decompositions.

\begin{lemma}\label{lem:exist_optimal_decomposition}
Let $x\in S(\M)$. There exist $y,z \in S(\M)$ such that $y + z = x$ and $\Norm{Ry}_1 + \Norm{Cz}_1 = m_1(x)$. 
\end{lemma}

\begin{proof}
Consider a sequence $(y^{(i)},z^{(i)})_{i\geq 0}$ such that:
$$\lim_{i\to\infty} \Norm{Ry^{(i)}}_1 + \Norm{Cz^{(i)}}_1 = m_1(x),$$ and for all $i\geq 0$, $y^{(i)} + z^{(i)} = x$.
By lemma $\ref{control:infty}$, we can suppose that the $y_{n}^{(i)}$ and $z_{n}^{(i)}$ are uniformly bounded in $L_{\infty}(\M)$. Recall that the sequence $x$ is finite, say of length $N$ and let:
$$e = s(\md{(Rx)^*}) \vee s(\md{Cx}).$$
By considering the sequences $(ey^{(i)}_ne)_{n\leq N}$ and $(ez^{(i)}e)_{n\leq N}$ we can also suppose that $y^{(i)}$ and $z^{(i)}$ are in $(e\M e)^N$.
Since the $y_{n}^{(i)}$ and $z_{n}^{(i)}$ are uniformly bounded in $L_{\infty}(e\M e)$, by the criterion of weak compacity in $L_1$ (\cite{Tak79}), up to extraction, we can suppose that for all $n\geq 0$, the sequence $(y_{n}^{(i)})_{i\geq 0}$ converges weakly in $L_1(e\M e)$ to an element $y_n\in e\M e$ and using Mazur's lemma, taking convex combinations we can even assume the norm-convergence in $L_1(\M)$. 

Let $y = (y_n)_{0\leq n \leq N}$ and $z = (z_n)_{0\leq n\leq N}$. Since the sequences $y$ and $z$ belong to $(e\M e)^N$, they belong to $S(\M)$. Note that $\{ a,b : a+b = x \}$ is closed and convex, so for all $n\geq 0$, $y_n +z_n = x_n$. Moreover, for $0\leq n\leq N$, we have $\lim_{i\to\infty} y^{(i)}_n \otimes e_{1,n} = y_n \otimes e_{1,n}$ in $L_1(e \M e \overline{\otimes} \B(\ell^2))$ and similarly $\lim_{i\to\infty} z^{(i)}_n \otimes e_{n,1} = z_n \otimes e_{n,1}$. Hence, by summing over $n$, we obtain:
 $m_1(x) = \lim_{i\to\infty} \Norm{Ry^{(i)}}_1 + \Norm{Cz^{(i)}}_1 = \Norm{Ry}_1 + \Norm{Cz}_1 $.

\end{proof}

\begin{theorem} \label{thm:factorization}
Let $x\in S(\M)$. There exist $\alpha,\beta \in \M_c^+$ and $u\in S(\M)$ such that $x = \alpha u + u \beta $, $s(\alpha) \leq \md{(Ru)^*} \leq 1$ and $s(\beta) \leq \md{Cu} \leq 1$.
\end{theorem}

\begin{proof}
Using the same notations as in the previous proof, we can work in $(e\M e)^{N}$ which garantees that all sequences considered are in $S(\M)$ and operators are finitely supported. Let $y,z\in S(\M)$ be the elements given by lemma \ref{lem:exist_optimal_decomposition} i.e $y + z = x$ and $\Norm{Ry}_1 + \Norm{Cz}_1 = m_1(x)$. Denote $\alpha = \md{(Ry)^*}$, $e = s(\alpha)$, $\beta = \md{Cz}$ and $f = s(\beta)$. Write $y = \alpha v$, $v\in S(\M)$ such that $\alpha Rv$ is a polar decomposition of $Ry$, in particular $v$ can be chosen such that $e v = v$. Similarly, $z = w \beta$ with $w\in S(\M)$ and $wf = w$.

By duality, there exists an element $u\in (e\M e)^N$ such that $\Norm{u}_{R_{\infty}\cap C_{\infty}} = 1$ and 
$$\tau(\Sum{i\geq 0}{} u_i^*x_i) = \Norm{x}_{R_1 + C_1} = m_1(x).$$
Let us rewrite the previous equality with the notations introduced previously:
\begin{align*}
m_1(x) &= \tau\big(\Sum{i\geq 0}{} u_i^*(\alpha v_i + w_i \beta)\big) \\
\tau(\alpha) + \tau(\beta) &= \tau\big(\Sum{i\geq 0}{} \alpha v_i u_i^* e\big) + \tau\big(\Sum{i\geq 0}{} f u_i^* w_i \beta\big)\\
\intertext{taking the real part on both sides of the equality, we get:}
\tau(\alpha) + \tau(\beta) &= \tau\big(\Re(\Sum{i\geq 0}{} \alpha v_i u_i^* e)\big) + \tau\big(\Re(\Sum{i\geq 0}{} f u_i^* w_i \beta)\big) \\
\intertext{and by the tracial property of $\tau$:}
\tau(\alpha) + \tau(\beta) &= \tau\big(\alpha\Re(\Sum{i\geq 0}{} v_i u_i^* e)\big) + \tau\big(\Re(\Sum{i\geq 0}{} f u_i^* w_i)\beta)\big).\\
\end{align*} 
Moreover, 
$$\Norm{\Re(\Sum{i\geq 0}{} v_i u_i^* e)}_{\infty} = \Norm{\Re(R(v)C(u^*e))}_{\infty} \leq 1$$
and since $ev = v$:
$$s(\Re(\Sum{i\geq 0}{} v_i u_i^* e)) = s(\Re(e(\Sum{i\geq 0}{} v_iu_i^*)e)) = s(e\Re (\Sum{i\geq 0}{} v_iu_i^*)e) \leq e$$
so $\Re(\Sum{i\geq 0}{} v_i u_i^* e) \leq e$. Similarly, since $wf = w$, $\Re(\Sum{i\geq 0}{} f u_i^* w_i) \leq f.$
Hence, we must have 
$$\Re(\Sum{i\geq 0}{} v_i u_i^* e) = e.$$
This means that
$$\md{\tau(R(v)C(u^*e))} = \Norm{R(v)}_2\Norm{C(u^*e)}_2.$$ 
There is equality in the Cauchy-Schwarz inequality in $L_2(\M \overline{\otimes} \B(\ell^2))$, so there exists $\lambda \in\Cb$, $eu = \lambda v$. The only possibility is that $e u = v$. Hence, $\alpha u = \alpha e u = \alpha v = y$. Similarly, $u \beta = z$. Therefore, $x = \alpha u + u \beta$. Let us now verify the other required properties. First, since $\Norm{u}_{R_{\infty}\cap C_{\infty}} = 1$, $\md{(Ru)^*} \leq 1$. Secondly, since $e u = v$, $e = \md{(R(eu))^*}$ and note that $\Norm{Ru}_{\infty} \leq 1$ implies that $\Norm{e(Ru)(Ru)^*}_2^2 \leq \tau(e)$. Moreover, by orthogonality,$\Norm{e(Ru)(Ru)^*}_2^2 = \Norm{e(Ru)(Ru)^*e}_2^2 + \Norm{e(Ru)(Ru)^*(1-e)}_2^2.$ Hence, $\tau (e) \geq \tau(e) + \Norm{e(Ru)(Ru)^*(1-e)}_2^2$, which means that $e(Ru)^*(Ru)(1-e) = 0$. Symmetrically, $(1-e)(Ru)(Ru)^*e = 0$ so 
$$(Ru)(Ru)^* - e = (Ru)(Ru)^* - e(Ru)(Ru)^*e = (1-e)(Ru)^*(Ru)(1-e) \geq 0.$$
By adjunction, we obtain similar inequalities for columns.
\end{proof}

We are now going to show that we can obtain Khintchine-type inequalities in a very general sense from the factorization found above. We will need the following inequality which is proved, up to some classical techniques using the Cayley transform in proposition $4.3$ of \cite{Ric18}. 

\begin{lemma} \label{lem:ineq_commutator}
Let $p\in (0,\infty)$, $q\in (0,\infty]$, $\theta \in (0,1)$. For all $t>0$, $\alpha,\beta \in \M_c^{+}$ and $b\in\M$ :
$$K_{t^\theta}(\alpha^{\theta}b + b\beta^{\theta},p/\theta,q/\theta) \les \bbox{K_t(\alpha b + b \beta ,p,q)}^{\theta}\Norm{b}_{\infty}^{1-\theta},$$
where the implicit constant only depends on $p,q$ and $\theta$.
\end{lemma}

An argument of how to deduce the previous lemma from \cite{Ric18} is given in the last section of this paper (see \ref{prop:ando8}, \ref{lem:cayleytransform} and the proof following right after). 

\begin{theorem} \label{thm:equivalence_Kt_alpha}
Let $x\in S(\M)$, $\alpha,\beta \in \M_c^+$ and $u\in S(\M)$ such that $x = \alpha u + u \beta $, $s(\alpha) \leq \md{(Ru)^*} \leq 1$ and $s(\beta) \leq \md{Cu} \leq 1$. Then, for all $t>0$ and $p>0$:
$$ K_t(\alpha,p,\infty) + K_t(\beta,p,\infty) \approx_c K_t(Gx,p,\infty),$$
where $c$ only depends on $p$ and $(\xi_i)$ (more precisely on the constant $c_{\infty}$ appearing in \eqref{equa:ineqKhininfty}).
\end{theorem}

\begin{proof}
\emph{Upper estimate.} Let $p>0$ and $t>0$. Note that $Gx = \alpha (Gu)+ (Gu)\beta$ where $\alpha$ is identified with $\alpha \otimes 1$ and $\beta$ with $\beta \otimes 1$ in $\M\overline{\otimes}\A$. By the Khintchine inequalities in $L_{\infty}$ (i.e. estimate \eqref{equa:ineqKhininfty}), since $\Norm{u}_{R_{\infty} \cap C_{\infty}} \leq 1$ we have $\Norm{Gu}_{\infty} \leq c_{\infty}$. Hence,
\begin{align*}
K_t(Gx,p,\infty) &\les K_t(\alpha (Gu),p,\infty) + K_t((Gu)\beta,p,\infty) \\
&\leq c_{\infty} (K_t(\alpha,p,\infty) + K_t(\beta,p,\infty)).
\end{align*}
\emph{Lower estimate.} Let $t>0$. We only prove the theorem for $p<1$, to obtain the result for $p \geq 1$ it suffices to take $\theta = 1$ in the argument.
Let us rewrite lemma \ref{lem:ineq_commutator}, with $\theta = p$, $q=\infty$:
$$K_{t^p}(\alpha^p(Gu) + (Gu)\beta^p ,1,\infty) \les \bbox{K_t(Gx,p,\infty)}^p \Norm{Gu}_{\infty}^{1-p}.$$
With this inequality, we can conclude without too much effort. Indeed:
$$K_{t^p}(\alpha^{p}(Gu) + (Gu)\beta^{p},1,\infty) \geq \Norm{Gu}_{\infty}^{-1}K_{t^p}(((Gu)\beta^p + \alpha^p(Gu))(Gu)^*,1,\infty).$$
By \eqref{equa:ineqKhininfty}, $\Norm{Gu}_{\infty} \les \Norm{u}_{R_{\infty} \cap C_{\infty}} = 1$. Let us now consider the conditional expectation $Id \otimes \tau_{\A}$.
\begin{align*}
(Id \otimes \tau_{\A})(((Gu)\beta^p + \alpha^p(Gu))(Gu)^*) &= (Id \otimes \tau_{\A})(\Sum{i,j\geq 0}{} u_i\beta^pu_j^* \otimes \xi_i\xi_j^* + \Sum{i,j\geq 0}{} \alpha^pu_iu_j^* \otimes \xi_i\xi_j^*)\\
&= \Sum{i=0}{\infty} u_i\beta^{p}u_i^* + \Sum{i=0}{\infty} \alpha^p u_iu_i^*.
\end{align*}
As a conditional expectation, $Id \otimes \tau_{\A}$ is bounded on $L_1$ and $L_{\infty}$ so
\begin{align*}
K_{t^p}(((Gu)\beta^p + \alpha^p(Gu))(Gu)^*,1,\infty) &\gtrsim K_{t^p}(\Sum{i=0}{\infty} u_i\beta^{p}u_i^* + \alpha^{p}(\Sum{i=0}{\infty} u_iu_i^*),1,\infty), \\
\intertext{note that $\Sum{i=0}{\infty} u_i\beta^{p}u_i^* \geq 0$, and that since $s(\alpha) \leq \Sum{i=0}{\infty} u_iu_i^* \leq 1$, we get $\alpha^{p}(\Sum{i=0}{\infty} u_iu_i^*) = \alpha^p \geq 0$. Hence,}
K_{t^p}(\Sum{i=0}{\infty} u_i\beta^{p}u_i^* + \alpha^{p}(\Sum{i=0}{\infty} u_iu_i^*),1,\infty) &\gtrsim K_{t^p}(\alpha^{p},1,\infty) \gtrsim K_t(\alpha,p,\infty)^{p},
\end{align*}
where we used the power theorem (proposition \ref{prop:powertheorem}) for the last inequality. The same tricks work for $\beta$ by multiplying $Gx$ by $(Gu)^*$ on the left.
\end{proof}

\begin{remark}
Theorem \ref{thm:intro:alpha} claimed in the introduction is obtained by a combination of the two previous theorems and the characterisation of interpolation spaces between $L_p$-spaces given by proposition \ref{prop:kfuncandinterpolation}.
\end{remark}

\begin{remark}
If we start with an $x\in S(\M)$ such that for all $i\geq 0$, $x_i = x_i^*$, the factorisation given by theorem \ref{thm:factorization} takes the following form. There exists $\alpha \in \M^+$ and $u\in S(\M)$ such that
\begin{itemize}
\item $x = u\alpha + \alpha u$,
\item for all $i\geq 0$, $u_i = u_i^*$,
\item $s(\alpha) \leq \Sum{i\geq 0}{} u_i^2 \leq 1$.
\end{itemize}
\end{remark}

\begin{remark}
From the results of this section, it is easy to recover the upper and lower Khintchine inequalities. More precisely, if $E$ is an interpolation space between $L_p$-spaces, $p\in (0,\infty]$ and $x\in S(\M)$:
$$\Norm{x}_{R_E + C_E} \les \Norm{Gx}_{E(\M_{\A})} \les \Norm{x}_{R_E \cap C_E}.$$
\end{remark}

\begin{proof}
Let $x\in S(\M)$. The left hand side inequality is obtained directly by considering the decomposition $y = \alpha u$ and $z = u \beta$ and applying the lower estimate in theorem \ref{thm:intro:alpha}.

To show the right hand side inequality, we make a computation similar to what appeared in the proof of theorem \ref{thm:equivalence_Kt_alpha}. First, using again theorem \ref{thm:intro:alpha}, we know that 
$$\Norm{Gx}_{E(\M_{\A})} \les \max (\Norm{\alpha}_{E(\M)}, \Norm{\beta}_{E(\M)}).$$
Moreover, note that 
$$\Norm{Rx}_{E(\M_{\B})} \gtrsim \Norm{Rx(Ru)^*}_{E(\M_{\B})} = \Norm{\alpha + (Ru)\beta(Ru)^*}_{E(\M_{\B})} \geq \Norm{\alpha}_{E(\M)}.$$
By adjunction, we also obtain
$$\Norm{Cx}_{E(\M_{\B})} \gtrsim \Norm{\beta}_{E(\M)}.$$
Hence,
$$\Norm{Gx}_{E(\M_{\A})} \les \max \p{\Norm{Rx}_{E(\M_{\B})}, \Norm{Cx}_{E(\M_{\B})}}.$$
\end{proof}

\section{A remark about martingale inequalities} \label{section:martingale inequalities}

In this section, we recover a variant of a result first proved in \cite{Narcisse2015} on martingale inequalities. The novelty compared to the original noncommutative martingale inequalities (\cite{PisXu97}) is that we show that the decomposition appearing for $1<p<2$ can be chosen to be independent of $p$. Note that this result has also been obtained, using a constructive approach, in a recent preprint (\cite{JiaRanWuZhu19}). The setting is the following, let $\F = (\M_n)_{n\geq 0}$ be a filtration on $\M$, and assume that for all $n\in\Nb$, the conditional expectation $\E_n: \M \to \M_n$ exists. Denote by $\Minf := \cup_{n\geq 0} \M_n$ the set of finite bounded martingales for the filtration $\F$. For any $x$ in $\Minf$, denote by $dx$ the associated martingale differences. 

\begin{theorem}
There exist constants $(k_p)_{p>1}$ such that for any $x\in\Minf$, there exists $y$ and $z\in\Minf$ such that $x = y + z$ and for all $p>1$:
$$\Norm{R(dy)}_p + \Norm{C(dz)}_p \leq k_p\Norm{x}_p.$$

\end{theorem}

\begin{proof}
By theorem $\ref{Main_1}$ applied for $p=1$, there exists $y'$ and $z'\in S(\M)$ such that $dx = y' + z'$ and for all $t\geq 0$, $K_t(Ry', 1, \infty) \leq C_1 K_t(Gx, 1, \infty)$ and $K_t(Cz', 1, \infty) \leq C_1 K_t(Gx, 1, \infty)$. By real interpolation, this means that for all $p>1$, $\Norm{Ry'}_p + \Norm{Cz'}_p \leq 2C_1 \Norm{Gx}_p$. Define $\Delta_0 = \E_0$ and for $n\geq 1$, $\Delta_n = \E_n - \E_{n-1}$. Let $dy = (\Delta_n(y'_n))_{n\geq 0}$ and $dz = (\Delta_n(z'_n))_{n\geq 0}$. This way $y$ and $z$ belong to $\Minf$ and they keep the same properties. Indeed, for all $n\in\Nb$, $dx_n = \Delta_n(dx_n) = \Delta_n(y'_n + z'_n) = dy_n + dz_n$ and by Stein's inequality (\citep{PisXu97}, \cite{Xu_notes}):
$$\Norm{R(dy)}_p + \Norm{C(dz)}_p \les \Norm{Ry'}_p + \Norm{Cz'}_p \les \Norm{G(dx)}_p.$$
Moreover, by considering the $\xi_i$ to be independent Rademacher variables and the unconditionality of martingale differences \cite{PisXu97}, $\Norm{G(dx)}_p \approx \Norm{x}_p$. Hence, there exists a constant $k_p$ (independant of $x$) such that:
$$\Norm{R(dy)}_p + \Norm{C(dz)}_p \leq k_p \Norm{x}_p.$$
\end{proof}

\begin{remark}
Since by interpolation, Burkholer-Gundy inequality stays true in all interpolation space $E$ between $L_p$-spaces, $\infty>p>1$, the argument above can be reproduced and the decomposition $y,z$ constructed in the proof verifies 
$$\Norm{R(dy)}_E + \Norm{C(dz)}_E \leq k_E \Norm{x}_E.$$
\end{remark}

\section{Counterexamples in $L_{2,\infty}$} \label{section:counterexample}

To complete the study of Khintchine inequalities in the motivating example of $L_{2,\infty}$, we provide a way to construct counterexamples and thus prove proposition \ref{prop:counterexample}. Explicit constructions can be made, but here we will use the Schur-Horn theorem which produces effortlessly a wide variety of examples for which the distributions of $Gx$ and $Rx$ can be prescribed.  In this section, we only consider $\M = \B(\ell^2)$ and the $\xi_i$ to be free Haar unitaries or Rademacher variables to make the computations explicit. 

A finite sequence $a = (a_1,...,a_n)$ will be identified with the infinite sequence $a = (a_1,...,a_n,0,0,0,...)$. To any sequence $a$ we associate the function 
$$f_a := \Sum{i=1}{\infty} a_i\Ind_{(i-1,i]}.$$

\begin{property}\label{prop:useSchurHorn}
Let $a$ and $b$ be two finite nonincreasing sequences of positive reals such that for all $n\in\Nb$,
$$\Sum{i=1}{n} a_i^2 \geq \Sum{i=1}{n} b_i^2\ \text{and}\ \Sum{i=1}{\infty} a_i^2 = \Sum{i=1}{\infty} b_i^2.$$
Then, there exists $x\in S(\M)$ such that
\begin{center}
$\mu(Gx) = \mu(Cx) = f_a$ and $\mu(Rx) = f_b$.
\end{center}
\end{property}

\begin{proof}
Let $N\in\Nb$ be the length of $a$ and $b$. The Schur-Horn theorem applied to $(a_i^2)$ and $(b_i^2)$ produces a positive matrix $M\in\Mb_N(\Cb)$ such that the eigenvalues of $M$ are given by $(a_i^2)$ and the diagonal of $M$ is given by $(b_i^2)$. Consider $x = (e_{i,i}M^{1/2})_{0\leq i \leq N}$. Then 
$$\md{Gx}^2 = \Sum{i,j}{} M^{1/2}e_{i,i}e_{j,j}M^{1/2} \otimes \xi_i^*\xi_j = \Sum{i}{} M^{1/2}e_{i,i}M^{1/2} \otimes 1 = M \otimes 1.$$ 
Similar computations give: $\md{Cx}^2 = M \otimes e_{1,1}$ and $\md{(Rx)^*}^2 = Diag(M) \otimes e_{1,1}$.
\end{proof}

\begin{proof}[Proof of proposition \ref{prop:counterexample}]
Fix $N\in \Nb$. Denote by $u_N$ the quantity $u_N := \Sum{i=1}{N} 1/i$. Let $a = (a_i)_{i\leq 1}$ and $b = (b_i)_{i\leq N}$ be defined by $a_1 = \sqrt{u_N}$, $a_i = 0$ for $i\geq 2$ and $b_i = \sqrt{1/i}$ for $i\leq N$, $b_i = 0$ for $i>N$. Applying the previous proposition, we obtain an element $x\in S(\M)$ such that $\Norm{Gx}_{2,\infty} = \sqrt{u_N} \approx \sqrt{\ln(N)}$ and $\Norm{Rx}_{2,\infty} = 1$. Hence, we cannot have $\Norm{.}_{\Hs_{2,\infty}} \les \Norm{.}_{R_{2,\infty} + C_{2,\infty}}$.

Now define $v_N := \Sum{i=1}{N} \Floor{N/i}$. Let $a = (\sqrt{\Floor{N/i}})_{i\leq N}$ and $b = (1)_{i\leq v_N}$. By applying, the previous proposition again, we obtain an element $x\in S(\M)$ such that $\Norm{Gx}_{2,\infty} = \sqrt{N}$ and $\Norm{Rx}_{2,\infty} = \sqrt{v_N} \approx \sqrt{N\ln(N)}$ which denies the possibility of having $\Norm{.}_{\Hs_{2,\infty}} \gtrsim \Norm{.}_{R_{2,\infty} \cap C_{2,\infty}}$.
\end{proof}

\begin{remark} \label{rem:counterexamplegeneral}
The flexibility given by proposition \ref{prop:useSchurHorn} means that the method can be applied to any symmetric space. It can be proven this way that for a symmetric space $E$ with the Fatou property:

$$ \Norm{.}_{\Hs_E} \approx \Norm{.}_{R_E \cap C_E} \Leftrightarrow E \in Int(L_2,L_{\infty}), $$

and if additionaly, $E$ is an interpolation space of $L_p$-spaces:

$$ \Norm{.}_{\Hs_E} \approx \Norm{.}_{R_E + C_E} \Leftrightarrow \exists p\in (0,2), E \in Int(L_p,L_{2}). $$

The remainder of the proofs of these results is essentially commutative and belongs to the classical theory of interpolation and function spaces. It can be found in \cite{Cad18Maj}. 
\end{remark}

\section{Technical lemmas} \label{section:technical lemmas}

\subsection{Related to the $K$-functional}

We start by proving lemma \ref{lem:matrixineq}.

\begin{proof}
\emph{i.} Define $b_{\e} = b+\e$. Set $c_{\e} = a^{1/2}b_{\e}^{-1/2}$. We only have to check that $c$ is a contraction.
$$c^*c = b_{\e}^{-1/2}a^{1/2}a^{1/2}b_{\e}^{-1/2} = b_{\e}^{-1/2}ab_{\e}^{-1/2} \leq b_{\e}^{-1/2}b_{\e}b_{\e}^{-1/2} = 1$$
Hence, $c_{\e}$ is a contraction. Let $c$ be a weak$^*$-limit of the $c_{\e}$. $c$ is a contraction and $a = cbc^*$.

\emph{ii.} If $a\leq b$ then $a^2 \leq a^{1/2}ba^{1/2}$. Now define $u$ to be a partial isometry appearing in the polar decomposition of $a^{1/2}b^{1/2}$ i.e $a^{1/2}b^{1/2} = u\md{a^{1/2}b^{1/2}}$. Then $u^*b^{1/2}a^{1/2}u = a^{1/2}b^{1/2}$ and $a^{1/2}ba^{1/2} = ub^{1/2}ab^{1/2}u^* \leq ub^2u^*$. 

\emph{iii.} For $\alpha \in [1,2]$, by operator convexity of the function $x\mapsto x^{\alpha}$, the result holds with $u=1$.
Now we proceed by induction. Suppose that the lemma in true for a $\alpha \geq 1$, we will show that it holds for  $\alpha' = 2\alpha$. By hypothesis, there exists a partial isometry $u$ such that:

$$(a + b)^{\alpha} \leq 2^{\alpha-1}u(a^{\alpha} + b^{\alpha})u^*.$$

And by ii. there exists $v$ such that:
\begin{align*}
(a + b)^{2\alpha} &\leq 2^{2\alpha-2}v(u(a^{\alpha} + b^{\alpha})u^*)^2 v^* \\
(a+b)^{\alpha'} &\leq 2^{\alpha'-2}vu(a^{\alpha} + b^{\alpha})^2(vu)^* \\
&\leq 2^{\alpha'-1}vu(a^{\alpha'} + b^{\alpha'})(vu)^{*} \\
\end{align*}
where the last inequality is a consequence of the operator convexity of $x \mapsto x^2$ or more precisely the inequality $(c + d)^2 \leq 2c^2 + 2d^2$ applied to $c = a^{\alpha}$ and $d = b^{\alpha}$.

\emph{iv.} Define $x = a + b$ and take two contractions $\alpha$ and $\beta$ such that $a^{1/2} = \alpha x^{1/2}$ and $b^{1/2} = \beta x^{1/2}$. We can suppose that $\alpha^* \alpha + \beta^* \beta = s(x)$. The operator $T : \M \to \M$ defined by $T(c) = \alpha c \alpha^*$ is positive and subunital. Hence, we can apply the Jensen inequality for the operator-concave function $t \mapsto t^{\theta}$ and we obtain $T(x)^{\theta} \geq T(x^{\theta})$. Similarly, $\beta x^{\theta} \beta^* \leq (\beta x \beta^*)^{\theta}$. Note also that, as we have used in the previous proof, there exist partial isometries $u$ and $v$ such that $x^{\theta /2}\alpha^*\alpha x^{\theta /2} = u\alpha x^{\theta} \alpha^* u^*$ and $x^{\theta /2}\beta^*\beta x^{\theta /2} = v \beta x^{\theta} \beta^* v^*$ (to see it, remark that by setting $y = \alpha x^{\theta /2}$ and $y = u^*\md{y}$ its polar decomposition, $x^{\theta /2}\alpha^*\alpha x^{\theta /2} = \md{y}^2$ and $\alpha x^{\theta} \alpha^* = u^*\md{y}^2 u$). Now, we can conclude by the following computation:

\begin{align*}
(a + b)^{\theta} &= x^\theta = x^{\theta /2}(\alpha^* \alpha + \beta^* \beta)x^{\theta /2} \\
&= x^{\theta /2}\alpha^*\alpha x^{\theta /2} + x^{\theta /2}\beta^*\beta x^{\theta /2} \\
&= u\alpha x^{\theta} \alpha^* u^* + v \beta x^{\theta} \beta^* v^* \\
&\leq u(\alpha x \alpha^*)^{\theta}u^* + v(\beta x \beta^*)^{\theta}v^* \\
&= ua^{\theta}u^* + vb^{\theta}v^*. \\
\end{align*}
\end{proof}

We now give the two lemmas used in the proof of proposition \ref{prop:Kt_independant of M}.

\begin{lemma}\label{lem:K_t approx finie}
Let $p,q\in (1,\infty]$ and $t>0$. Let $x\in L_p(\M) + L_q(\M)$, then
$$K_t(x,L_p(\M),L_q(\M)) = \supb{e\in\Pc_c(\M)} K_t(exe,L_p(\M),L_q(\M)).$$
\end{lemma}

\begin{proof}
Define $E = L_{p'}(\M) \cap t^{-1}L_{q'}(\M)$ with $p' = (1-p^{-1})^{-1}$ and $q' = (1-q^{-1})^{-1}$ and note that $L_p(\M) + tL_q(\M) = E^*$. Since $p'$ and $q'$ belong to $[1,\infty)$, $\M_c$ is dense in $E$. Hence,
\begin{align*}
K_t(x,L_p(\M),L_q(\M)) 
&= \sup \{\md{\tau(xy)} : y \in \M_c, \Norm{y}_{E} = 1\} \\
&= \sup \{\md{\tau(xeye)} : e \in \Pc_c(\M), y \in \M_c, \Norm{y}_{E} = 1\} \\
&= \supb{e\in\Pc_c(\M)} \Norm{exe}_{L_p(\M) + tL_q(\M)}\\
&= \supb{e\in\Pc_c(\M)} K_t(exe,L_p(\M),L_q(\M)).
\end{align*}
\end{proof}

\begin{lemma}\label{lem:tech_approx_mu_finite}
Let $x\in L_0(\M)$, $\epsilon >0$ and $f\in \Pc_c(L_{\infty}(0,\infty))$. There exists a finite projection $e\in \M$ such that:
$$\mu(\mu(x)f) \leq \mu(exe) + \e.$$
\end{lemma}

\begin{proof}
Let $f\in \Pc_c(L_{\infty}(0,\infty))$ such that $\Int{\Rb^+}{} f = T$, then 
$$\mu(\mu(x)f) \les \mu(x)\Ind_{(0,T)} = \mu(\mu(x)\Ind_{(0,T)}).$$
Hence, it suffices to consider projections of the form $\Ind_{(0,T)}$, $T>0$. Let $a = \mu_T(x)$ and let $e_1 = \Ind_{(a,\infty)}(\md{x})$. By definition of $\mu$, $e_1$ is finite, write $t_1 = \tau(e_1)$ and note that $\mu_t(x) = a$ if $t_1\leq t \leq T$. Since $a = \mu(T)$, $\tau(\Ind_{(a-\e,\infty)}) \geq T$, hence $\tau(\Ind_{(a-\e,a]}) \geq T-t_1$. Let $e_2$ be a finite projection such that $e_2 \leq \Ind_{(a-\e,a]}$ and $\tau(e_2) \geq T-t_1$. It is possible to find such a projection since $\tau$ is semifinite. Define $e = e_1 + e_2$. By construction, for $t \in (0,t_1)$, $\mu_t(exe) = \mu_t(x) = f(t)\mu_t(x)$ and for $t\in [t_1,T)$, $\mu_t(exe) > a - \e = f(t)\mu_t(x) - \e$. So $e$ satisfies the conditions of the lemma.
\end{proof}

\subsection{Rows and columns}

Let us now prove the lemma that we used several times when manipulating rows and columns.

\begin{lemma} \label{lem:trick}
Let $p>0$, $y,z\in S(\M)$ and $e,f\in \Pc(\M)$ such that $e$ commutes with $\md{(Ry)^*}$. Then,
\begin{enumerate}[label=\text{\roman*.}]
\item $e\md{(Ry)^*}^p = \md{(R(ey))^*}^p,$ 
\item $\Norm{Ry}_p^p = \Norm{R(ey)}_p^p + \Norm{R(e^{\bot}y)}_p^p,$
\item $\md{C(fz)}^2 \leq \md{Cz}^2$ and consequently $\Norm{C(fz)}_p \leq \Norm{Cz}_p.$
\end{enumerate}
\end{lemma}

\begin{proof}
i. This is a direct computation:
\begin{align*}
e\md{(Ry)^*}^p &= e^{p/2}(\md{(Ry)^*}^2)^{p/2} = (e\md{(Ry)^*}^2)^{p/2} = (e(\Sum{i}{} y_iy_i^*)e)^{p/2} \\
&= (\Sum{i}{} (ey_i)(ey_i)^*)^{p/2} = \md{R(ey)^*}^p. 
\end{align*}
ii. Write $f := e^{\bot}$. Using i., this is again a direct computation:
\begin{align*}
\Norm{Ry}_p^p &= \tau(\md{(Ry)^*}^p) = \tau((e+f)\md{(Ry)^*}^p) = \tau(e\md{(Ry)^*}^p)+\tau(f\md{(Ry)^*}^p)\\
 &= \tau(\md{(R(ey))^*}^p)+\tau(\md{(R(fy))^*}^p) = \Norm{R(ey)}_p^p + \Norm{R(fy)}_p^p.\\
\end{align*}

iii. Indeed, $\md{C(fz)}^2 = \Sum{i}{} z_i^*fz_i \leq \Sum{i}{} z_i^*z_i = \md{Cz}^2.$ 
\end{proof}

\subsection{$K$-functional of commutators}

Now, we prove lemma \ref{lem:ineq_commutator}. We recall the following proposition that is the main result of \cite{Ric18} (proposition $4.3$). 

\begin{property} \label{prop:ando8}
Let $0<p<q\leq \infty$, $\theta \in (0,1)$ and $x,y\in (L_p(\M) + L_q(\M))^{sa}$. Then for all $t>0$ and $f: x \mapsto \md{x}^{\theta}$ or $f: x \mapsto \text{sgn}(x)\md{x}^{\theta}$, we have:
$$ K_{t^{\theta}}(f(x) - f(y),p/\theta,q/\theta) \les K_t(x-y,p,q)^{\theta}.$$
\end{property}

The relationship between Mazur maps and anti-commutators has already been explicited in \cite{Ric15}. The arguments rely on $2 \times 2$ matrix tricks and the Cayley transform. We reproduce them in our context for completion but the following proofs do not contain any new idea. We break it down into two steps, the first one is given by the following lemma and uses the Cayley transform.

\begin{lemma} \label{lem:cayleytransform}
Let $p\in (0,\infty)$, $q\in (0,\infty]$, $\theta \in (0,1)$. For all $t>0$, $f: x \mapsto \md{x}^{\theta}$ or $f: x \mapsto \text{sgn}(x)\md{x}^{\theta}$, $\alpha \in \M^{sa}$ and $b\in\M^{sa}$ :
$$K_{t^\theta}(bf(\alpha) - f(\alpha)b,p/\theta,q/\theta) \les \bbox{K_t(b\alpha - \alpha b,p,q)}^{\theta}\Norm{b}_{\infty}^{1-\theta},$$
where the implicit constant only depends on $p,q$ and $\theta$.
\end{lemma}

\begin{proof}
We can assume that $\Norm{b}_{\infty} = 1$ by homogeneity. Since $b = b^*$, the Cayley transform is defined by
$$u = (b+i)(b-i)^{-1},\; b = 2i(1-u)^{-1} - i.$$
Note that $u$ is unitary and by the second formula $\Norm{(1-u)^{-1}}_{\infty} \leq \frac{1}{\sqrt{2}}$. These facts allow us to conclude by a computation using \ref{prop:ando8}.
\begin{align*}
K_{t^\theta}(bf(\alpha) - f(\alpha)b,p/\theta,q/\theta) &\leq 
2K_{t^\theta}((1-u)^{-1}f(\alpha) - f(\alpha)(1-u)^{-1},p/\theta,q/\theta) \\
&\leq 2\Norm{(1-u)^{-1}}_{\infty}^2 K_{t^\theta}((1-u)f(\alpha) - f(\alpha)(1-u),p/\theta,q/\theta) \\
&\leq K_{t^\theta}(f(u^*\alpha u) - f(\alpha),p/\theta,q/\theta) 
\intertext{using proposition \ref{prop:ando8},}
&\les K_t(\alpha u - u \alpha ,p,q)^{\theta} \\
&\les \Norm{(b-i)^{-1}}_{\infty}^{2\theta} \bbox{K_t((b-i)\alpha(b+i) - (b+i)\alpha (b-i),p,q)}^{\theta} \\
&\les K_t(\alpha b - b \alpha,p,q)^{\theta}.
\end{align*}
\end{proof}

We can now conclude the main proof using a matrix trick.

\begin{proof}[Proof of lemma \ref{lem:ineq_commutator}]
Let $b\in \M$ and $\alpha,\beta \in \M^+$. Consider:
$$ b' = \begin{pmatrix}
    0  & b \\
    b^*  & 0  \\

\end{pmatrix}
,\; \alpha' = \begin{pmatrix}
    \alpha & 0	 \\
    0      & -\beta  \\
\end{pmatrix}.$$
Note that in general, the matrix $b'$ has the same distribution as $b \oplus b$. Hence, $\Norm{b'}_{\infty} = \Norm{b}_{\infty}$ and for all $t>0$,
$$K_t(b',p,q) \approx K_{t}(b,p,q).$$ 
Indeed, $x \mapsto K_t(x,p,q)$ is a quasi-norm so $K_t(b,p,q) \les K_{t}(b',p,q)$ and $\mu(b') = \mu(b \oplus b) \geq \mu(b)$ so $K_t(b',p,q) \geq K_t(b,p,q)$.
It is now straightforward to check that for all $t>0$,
\[ 
K_t(\alpha b + b \beta, p,q) \approx K_{t}(\alpha' b' - b' \alpha',p,q) \addtag\label{eq:Kt2x2matrix}
\]
and that for $f: x \mapsto \text{sgn}(x)\md{x}^{\theta}$, 
\[
K_t(f(\alpha) b + b f(\beta), p,q) \approx K_{t}(f(\alpha') b' - b' f(\alpha'),p,q). \addtag\label{eq:Kt2x2matrix_with_f}
\]
By lemma \ref{lem:cayleytransform}  applied to $b',\alpha',\theta$:
$$K_{t^\theta}(b'f(\alpha') - f(\alpha')b',p/\theta,q/\theta) \les \bbox{K_t(b'\alpha' - \alpha'b',p,q)}^{\theta}\Norm{b'}_{\infty}^{1-\theta}.$$
Hence, we can conclude by \eqref{eq:Kt2x2matrix} and \eqref{eq:Kt2x2matrix_with_f} and obtain:
$$K_{t^\theta}(bf(\alpha) - f(\alpha)b,p/\theta,q/\theta) \les \bbox{K_t(b\alpha - \alpha b,p,q)}^{\theta}\Norm{b}_{\infty}^{1-\theta}.$$
\end{proof}

\subsection*{Acknowledgements}

I am very grateful to my advisor Eric Ricard for many valuable discussions, his guidance throughout the making of this article, and in particular in developing section \ref{Section alpha}. I would also like to thank Fedor Sukochev and Dmitriy Zanin for enlightning exchanges on section \ref{section:counterexample} and interpolation theory.

\bibliographystyle{plain}
\bibliography{WKbib}

\begin{thebibliography}{}

\bibitem[Bergh and L\"{o}fstr\"{o}m, 1976]{BerghLofstrom}
Bergh, J. and L\"{o}fstr\"{o}m, J. (1976).
\newblock {\em Interpolation spaces. {A}n introduction}.
\newblock Springer-Verlag, Berlin-New York.
\newblock Grundlehren der Mathematischen Wissenschaften, No. 223.

\bibitem[Cadilhac, 2018a]{Cad18Maj}
Cadilhac, L. (2018a).
\newblock Majorization, interpolation and noncommutative {K}hintchine
  inequalities.
\newblock {\em preprint}.

\bibitem[Cadilhac, 2018b]{Cad18}
Cadilhac, L. (2018b).
\newblock Weak boundedness of {C}alder\'on-{Z}ygmund operators on
  noncommutative {$L_1$}-spaces.
\newblock {\em J. Funct. Anal.}, 274(3):769--796.

\bibitem[Cwikel, 1981]{Cwi81}
Cwikel, M. (1981).
\newblock Monotonicity properties of interpolation spaces. {II}.
\newblock {\em Ark. Mat.}, 19(1):123--136.

\bibitem[Dirksen, 2015]{Dir15}
Dirksen, S. (2015).
\newblock Noncommutative {B}oyd interpolation theorems.
\newblock {\em Trans. Amer. Math. Soc.}, 367(6):4079--4110.

\bibitem[Dirksen et~al., 2011]{DirSuk11}
Dirksen, S., de~Pagter, B., Potapov, D., and Sukochev, F. (2011).
\newblock Rosenthal inequalities in noncommutative symmetric spaces.
\newblock {\em J. Funct. Anal.}, 261(10):2890--2925.

\bibitem[Dirksen and Ricard, 2013]{RicardDirksen}
Dirksen, S. and Ricard, E. (2013).
\newblock Some remarks on noncommutative {K}hintchine inequalities.
\newblock {\em Bull. Lond. Math. Soc.}, 45(3):618--624.

\bibitem[Fack and Kosaki, 1986]{FackKosaki}
Fack, T. and Kosaki, H. (1986).
\newblock Generalized {$s$}-numbers of {$\tau$}-measurable operators.
\newblock {\em Pacific J. Math.}, 123(2):269--300.

\bibitem[Haagerup, 1978]{Haa78}
Haagerup, U. (1978).
\newblock An example of a nonnuclear {$C^{\ast} $}-algebra, which has the
  metric approximation property.
\newblock {\em Invent. Math.}, 50(3):279--293.

\bibitem[Haagerup and Pisier, 1993]{HaaPis93}
Haagerup, U. and Pisier, G. (1993).
\newblock Bounded linear operators between {$C^*$}-algebras.
\newblock {\em Duke Math. J.}, 71(3):889--925.

\bibitem[Jiao et~al., 2019]{JiaRanWuZhu19}
Jiao, Y., Randrianantoanina, N., Wu, L., and Zhu, D. (2019).
\newblock Square functions for noncommutative differentially subordinate
  martingales.
\newblock {\em arXiv:1901.08752}.

\bibitem[Junge, 2002]{Jun02}
Junge, M. (2002).
\newblock Doob's inequality for non-commutative martingales.
\newblock {\em J. Reine Angew. Math.}, 549:149--190.

\bibitem[Junge and Xu, 2003]{JunXu03}
Junge, M. and Xu, Q. (2003).
\newblock Noncommutative {B}urkholder/{R}osenthal inequalities.
\newblock {\em Ann. Probab.}, 31(2):948--995.

\bibitem[Junge and Xu, 2008]{JunXu08}
Junge, M. and Xu, Q. (2008).
\newblock Noncommutative {B}urkholder/{R}osenthal inequalities. {II}.
  {A}pplications.
\newblock {\em Israel J. Math.}, 167:227--282.

\bibitem[Kalton and Sukochev, 2008]{KalSuk08}
Kalton, N.~J. and Sukochev, F.~A. (2008).
\newblock Symmetric norms and spaces of operators.
\newblock {\em J. Reine Angew. Math.}, 621:81--121.

\bibitem[Kre\u\i~n et~al., 1982]{KrePetSem82}
Kre\u\i~n, S.~G., Petun\=\i~n, Y.~I., and Sem\"enov, E.~M. (1982).
\newblock {\em Interpolation of linear operators}, volume~54 of {\em
  Translations of Mathematical Monographs}.
\newblock American Mathematical Society, Providence, R.I.
\newblock Translated from the Russian by J. Sz\H ucs.

\bibitem[Le~Merdy and Sukochev, 2008]{LeMSuk08}
Le~Merdy, C. and Sukochev, F. (2008).
\newblock Rademacher averages on noncommutative symmetric spaces.
\newblock {\em J. Funct. Anal.}, 255(12):3329--3355.

\bibitem[Lust-Piquard, 1986]{Lus86}
Lust-Piquard, F. (1986).
\newblock In\'egalit\'es de {K}hintchine dans {$C_p\;(1<p<\infty)$}.
\newblock {\em C. R. Acad. Sci. Paris S\'er. I Math.}, 303(7):289--292.

\bibitem[Lust-Piquard and Pisier, 1991]{LusPis91}
Lust-Piquard, F. and Pisier, G. (1991).
\newblock Noncommutative {K}hintchine and {P}aley inequalities.
\newblock {\em Ark. Mat.}, 29(2):241--260.

\bibitem[Lust-Piquard and Xu, 2007]{LusXu07}
Lust-Piquard, F. and Xu, Q. (2007).
\newblock The little {G}rothendieck theorem and {K}hintchine inequalities for
  symmetric spaces of measurable operators.
\newblock {\em J. Funct. Anal.}, 244(2):488--503.

\bibitem[Mei and Parcet, 2009]{MeiPar09}
Mei, T. and Parcet, J. (2009).
\newblock Pseudo-localization of singular integrals and noncommutative
  {L}ittlewood-{P}aley inequalities.
\newblock {\em Int. Math. Res. Not. IMRN}, (8):1433--1487.

\bibitem[Parcet, 2009]{Par09}
Parcet, J. (2009).
\newblock Pseudo-localization of singular integrals and noncommutative
  {C}alder\'on-{Z}ygmund theory.
\newblock {\em J. Funct. Anal.}, 256(2):509--593.

\bibitem[Parcet and Pisier, 2005]{ParPis05}
Parcet, J. and Pisier, G. (2005).
\newblock Non-commutative {K}hintchine type inequalities associated with free
  groups.
\newblock {\em Indiana Univ. Math. J.}, 54(2):531--556.

\bibitem[Pisier, 1992]{Pis92}
Pisier, G. (1992).
\newblock Interpolation between {$H^p$} spaces and noncommutative
  generalizations. {I}.
\newblock {\em Pacific J. Math.}, 155(2):341--368.

\bibitem[Pisier, 2009]{Pis09}
Pisier, G. (2009).
\newblock Remarks on the non-commutative {K}hintchine inequalities for
  {$0<p<2$}.
\newblock {\em J. Funct. Anal.}, 256(12):4128--4161.

\bibitem[Pisier, 2011]{Pis11}
Pisier, G. (2011).
\newblock Real interpolation between row and column spaces.
\newblock {\em Bull. Pol. Acad. Sci. Math.}, 59(3):237--259.

\bibitem[Pisier and Ricard, 2017]{PisierRicard2017}
Pisier, G. and Ricard, E. (2017).
\newblock The non-commutative {K}hintchine inequalities for {$0<p<1$}.
\newblock {\em J. Inst. Math. Jussieu}, 16(5):1103--1123.

\bibitem[Pisier and Xu, 1997]{PisXu97}
Pisier, G. and Xu, Q. (1997).
\newblock Non-commutative martingale inequalities.
\newblock {\em Comm. Math. Phys.}, 189(3):667--698.

\bibitem[Randrianantoanina and Wu, 2015]{Narcisse2015}
Randrianantoanina, N. and Wu, L. (2015).
\newblock Martingale inequalities in noncommutative symmetric spaces.
\newblock {\em J. Funct. Anal.}, 269(7):2222--2253.

\bibitem[Ricard, 2015]{Ric15}
Ricard, E. (2015).
\newblock H\"older estimates for the noncommutative {M}azur maps.
\newblock {\em Arch. Math. (Basel)}, 104(1):37--45.

\bibitem[Ricard, 2018]{Ric18}
Ricard, E. (2018).
\newblock Fractional powers on noncommutative {$L_p$} for {$p<1$}.
\newblock {\em Adv. Math.}, 333:194--211.

\bibitem[Ricard and Xu, 2006]{RicXu06}
Ricard, E. and Xu, Q. (2006).
\newblock Khintchine type inequalities for reduced free products and
  applications.
\newblock {\em J. Reine Angew. Math.}, 599:27--59.

\bibitem[Sparr, 1978]{Spa78}
Sparr, G. (1978).
\newblock Interpolation of weighted {$L_{p}$}-spaces.
\newblock {\em Studia Math.}, 62(3):229--271.

\bibitem[Takesaki, 1979]{Tak79}
Takesaki, M. (1979).
\newblock {\em Theory of operator algebras. {I}}.
\newblock Springer-Verlag, New York-Heidelberg.

\bibitem[Xu, 2011]{Xu_notes}
Xu, Q. (2011).
\newblock Noncommutative ${L}_p$-spaces.
\newblock unpublished.

\end{thebibliography}

\
{\footnotesize%

 \textsc{Normandie Univ, UNICAEN, CNRS, Laboratoire de math\'ematiques Nicolas Oresme, 14000 Caen, France} \par
 \textit{E-mail address}: \texttt{leonard.cadilhac@unicaen.fr} 
}

\end{document}